\documentclass[a4paper, 11pt]{article}
\usepackage{amsmath, amssymb,amscd, amsthm}
\usepackage[mathscr]{eucal}
\usepackage{graphics}
\usepackage{fullpage}
\usepackage{url}
\newcommand\cyr{%
 \renewcommand\rmdefault{wncyr}%
 \renewcommand\sfdefault{wncyss}%
 \renewcommand\encodingdefault{OT2}%
\normalfont\selectfont} \DeclareTextFontCommand{\textcyr}{\cyr}

\newtheorem{theorem}{Theorem}
\newtheorem{lemma}[theorem]{Lemma}
\newtheorem{corollary}[theorem]{Corollary}
\newtheorem{proposition}[theorem]{Proposition}
\newtheorem{remark}[theorem]{Remark}

\newtheorem{definition}[theorem]{Definition}

\newtheorem*{thm1}{Theorem 1}

\begin{document}

\title{The Square Variation of Rearranged Fourier Series}
\author{Allison Lewko \ \and Mark Lewko\thanks{M. Lewko is supported by a NSF Postdoctoral Fellowship, DMS-1204206.}}
\date{}
\maketitle
\begin{abstract}We prove that there exists a rearrangement of the first $N$ elements of the trigonometric system such that the $L^2$-norm of the square variation operator is at most $O_{\epsilon}(\log^{9/22+\epsilon}(N))$. This is an improvement over $O(\log^{1/2}(N))$ from the canonical ordering.
\end{abstract}
\section{Introduction}
Let $\Phi := \{\phi_n\}_{n=1}^{N}$ denote an orthonormal system (ONS) of functions from a probability space, $\mathbb{T}$, to $\mathbb{R}$. One is often interested, usually motivated by questions regarding almost everywhere convergence, in the behavior of the maximal partial sum operator
\[\mathcal{M}f(x) := \max_{\ell \leq N}\left| \sum_{n=1}^{\ell} a_n \phi_n(x)  \right|. \]
For an arbitrary ONS, the Rademacher-Menshov theorem states that $||\mathcal{M}f ||_{L^2} \ll \log(N) ||f||_{L^2}$, where the $\log(N)$ factor is known to be sharp. However, one can do much better for many classical systems, for instance the well-known Carleson-Hunt inequality allows one to replace $\log(N)$ with an absolute constant in the case of the trigonometric system.  More recently, there has been interest in variational refinements of these maximal results. We define the $r$-variation operator
\[  \mathcal{V}^{r}f(x)  := \left(\max_{\pi \in \mathcal{P}_{N}} \sum_{I \in \pi } \left| \sum_{n\in I} a_n\phi_n(x) \right|^r \right)^{1/r} \]
where $\mathcal{P}_N$ denotes the set of partitions of $[N]$ into subintervals.  Clearly, $\mathcal{V}^r f$ is pointwise non-decreasing as $r$ decreases, and $\mathcal{M}f  \leq \mathcal{V}^{r}f $ for all $r< \infty$. While estimates involving the maximal operator typically imply statements regarding almost everywhere convergence, estimates involving the larger variational operators typically imply quantitative statements about the rate of convergence. In \cite{LewkoO}, it was shown that the Rademacher-Menshov theorem can be strengthened to $|| \mathcal{V}^{2}f ||_{L^2} \ll \log(N)||f||_{L^2}$.

In the case of the trigonometric system, it was shown by Oberlin, Seeger, Tao, Thiele, and Wright \cite{varCarleson} that $||\mathcal{V}^{r}f||_{2} \ll_{r} ||f||_{2} $ for $r>2$ (strengthening the Carleson-Hunt inequality). In the case $r=2$, one can deduce the inequality $||\mathcal{V}^{2}f||_{2} \ll \sqrt{\log(N)}||f||_{2} $ from the Carleson-Hunt inequality (see \cite{LewkoO}, Theorem 3). Moreover, the $\sqrt{\log(N)}$ can be shown to be sharp (see \cite{varCarleson}, section 2).

When considering questions regarding partial sums of an ONS, the ordering of the system plays a crucial role. For instance, Olevskii \cite{Olevskii} has shown that any infinite complete ONS can be reordered in a manner such that almost everywhere convergence fails for some $L^2$ function. The related question of whether an ONS can be reordered in a manner such that almost everywhere convergence holds for every $L^2$ function is known as Kolmogorov's rearrangement problem and is one of the central open problems regarding orthonormal systems. Going further, Garsia has conjectured \cite{GarsiaBook} (see also \cite{GarsiaRe}) that given an ONS $\Phi := \{\phi_n\}_{n=1}^{N}$, one may find a permutation $\sigma(n):[N] \rightarrow [N]$ such that the reordered system $\{\phi_{\sigma(n)}\}_{n=1}^{N}$ satisfies $||\mathcal{M}f||_{2} \ll ||f||_{2}$ where the implied constant is absolute and independent of the system. This is known to imply an affirmative solution to Kolmogorov's problem.  As partial progress towards Garsia's conjecture, Bourgain \cite{Bour} has shown (for uniformly bounded systems) that one may always find a permutation of $[N]$ such that $||\mathcal{M}f||_{2} \ll \log \log (N) ||f||_{2}$. In \cite{LewkoO}, this was strengthened to $||\mathcal{V}^r f||_{2} \ll_{r} \log \log (N) ||f||_{2}$ for $r>2$. While these estimates have strong consequences for very general orthonormal systems, their conclusions are weaker than what is known to be true for most classical systems (such as the trigonometric system) in their canonical orderings.

With this in mind, it was asked in \cite{LewkoO} if one could improve the inequality $||\mathcal{V}^{2}f||_{2} \ll \sqrt{\log(N)}||f||_{2}$ by reordering the first $N$ elements of the trigonometric system. Here we provide an affirmative answer to this question by proving the following theorem:
\begin{thm1}\label{thm:MainI}
Let $\Phi := \{ \phi_n \}_{n=1}^{N}$ denote a ONS\footnote{We have defined an ONS to be real-valued, however the result follows for the complex-valued trigonometric system by applying the result to the real and imaginary parts separately.} uniformly bounded by some constant $A$. Let $\epsilon >0 , \gamma >1$. Then, with probability at least $1- c N^{-\gamma}$ (for some universal $c$), for a uniformly random permutation $\sigma: [N] \rightarrow [N]$, the system $\{ \phi_{\sigma(n)} \}_{n=1}^{N}$ will satisfy
$$|| \mathcal{V}^2 f ||_{2} \ll_{A,\epsilon, \gamma} \log^{\frac{9}{22} + \epsilon}(N) ||f||_{2}.$$
\end{thm1}

It turns out that treating the $\mathcal{V}^2$ operator requires a considerably more delicate analysis than the maximal and $r$-variation ($r>2$) cases previously studied. Indeed, the Dudley-type chaining/covering number methods used in \cite{Bour} and \cite{LewkoO} alone seem incapable of achieving anything better than $|| \mathcal{V}^2 f ||_{L^2} \ll \log^{1/2}(N)\log\log(N)  ||f||_{L^2}$ (see \cite{LewkoO}). The limitations of these methods have been previously overcome in the context of related problems, most notably Bourgain's solution to the $\Lambda(p)$-problem \cite{Bour}, as well as Talagrand's alternate approach and generalizations \cite{TalagrandSmooth}. The probabilistic component of our current work will use tools from both Bourgain and Talagrand's works, however additional complications enter as we will need to work with random subsets of a much greater density and derive stronger concentration bounds (see Section \ref{sec:remarks} for further discussion of these issues and an overview of this part of our argument). A careful combinatorial organization is also needed to reduce estimates for the $\mathcal{V}^2$ operator to questions amenable to these probabilistic methods. Here we rely, in part, on ideas from Taylor's work \cite{Taylor} on the path variation of Brownian motion.

We do not expect that the exponent $\frac{9}{22}$ is sharp. In the maximal case, it is known that Bourgain's estimate $||\mathcal{M}f||_{2} \ll \log \log (N) ||f||_{2}$ is the best one can achieve with a purely probabilistic argument (see remark 2 in \cite{Bour}). It is consistent with our knowledge that probabilistic arguments might be able to achieve $||\mathcal{V}^2 f||_{2} \ll \log \log (N) ||f||_{2}$, although this would surely require a much deeper analysis. In \cite{LewkoO} (see Theorem 6) it was shown that for any bounded ONS one may find a function $f$ such that $||\mathcal{V}^2 f||_{2} \gg \sqrt{\log \log (N)} ||f||_{2}$, which gives a lower bound for any ordering. This fact is closely related to the law of the iterated logarithm (see \cite{LewkoProb}).

\section{Preliminaries}
We use the notation $ x\ll y$, for instance, to mean that there exists an absolute constant $C$ such that $x \leq C y$. We similarly employ notation like $x \ll_p y$ to mean that there exists a constant $C_p$ depending only on $p$ such that $x \leq C_p y$. We will use $[N]$ to denote the set of the first $N$ natural numbers $\{1,2,\ldots,N\}$.

\begin{remark}\label{rem:bounded}Throughout this paper, we will consider orthonormal systems $\Phi:=\{\phi_n\}_{n=1}^{N}$ uniformly bounded by a fixed constant $A$, meaning that $|\phi_n(x)| \leq A$ for all $n \in [N]$ and all $x \in \mathbb{T}$. We will refer to these simply as \textbf{bounded orthonormal systems}. Since $A$ is fixed, we allow dependence on $A$ in all implicit constants (such as the asymptotic notations $\ll$ and $O(\cdot)$) which we will not always explicitly state.
\end{remark}

We let $\Gamma$ denote a convex, symmetric function $\Gamma: \mathbb{R} \rightarrow \mathbb{R}^+$ such that $\Gamma(0)=0$, $\Gamma$ is increasing on $\mathbb{R}^+$, and $\Gamma(t)$ tends to infinity as $t$ tends to infinity. The (Luxemburg) Orlicz space norm associated to $\Gamma$ is then defined by:

\begin{definition}\label{def:GammaNorm}For $f: \mathbb{T} \rightarrow \mathbb{R}$,
\[ ||f||_{\Gamma_K} := \inf \left\{ \gamma>0 | \int_{\mathbb{T}} \Gamma_K \left(\frac{f}{\lambda}\right) \leq 1\right\}.\]
\end{definition}

We define the following convex function from $\mathbb{R}$ to $\mathbb{R}$, parameterized by a value $1 < K < \infty$ and a value $ 2 < p <3$:
\[\Gamma_K(t) := \left\{
                   \begin{array}{ll}
                     |t|^p, & \hbox{if $|t| \leq K$;} \\
                     \left(1 + \frac{p-2}{2}\right)K^{p-2}t^2 - \left(\frac{p-2}{2}\right) K^p, & \hbox{if $|t|> K$.}
                   \end{array}
                 \right.\]
We also define
\[ \gamma_K(t) := \left\{
                    \begin{array}{ll}
                      |t|^{p-2}, & \hbox{if $|t| \leq K$;} \\
                      K^{p-2}, & \hbox{if $|t| > K$.}
                    \end{array}
                  \right.\]
We observe that $t^2 \gamma_K(t) \leq \Gamma_K(t) \leq \left(1 + \frac{p-2}{2}\right)t^2 \gamma_K(t)$.

\begin{lemma}\label{lem:gammabasic} For any fixed $1 < K < \infty$ and $2 < p <3$, it holds for all $s, t \in \mathbb{R}$ that
\[ |s \gamma_K(s) - t \gamma_K(t) | \leq 3(\gamma_K(s) + \gamma_K(t)) |s-t|.\]
\end{lemma}

\begin{proof}
Without loss of generality, we may assume that $|s| \geq |t|$. We define $\beta$ by $t = \beta s$, where $|\beta| \leq 1$. For notational concision, we also define $\alpha = p-2$. We first consider the case where $|s|, |t| \leq K$. In this case, we consider the quantity
\[ |s \gamma_K (s) - t \gamma_K(t)| = |s|^{1+\alpha} \left(1 - \beta |\beta|^\alpha\right).\]
We must establish that this is upper bounded by $3|s|^{1+\alpha}(1 - \beta)(1+|\beta|^\alpha)$. Since $|\beta|\leq 1$ and $\alpha \leq 1$, we have $|\beta|^\alpha \geq |\beta|$, and hence $(1-\beta)(1 + |\beta|^\alpha) = 1 - \beta + |\beta|^\alpha - \beta |\beta|^\alpha \geq 1 - \beta |\beta|^\alpha$. The desired inequality follows (note that the factor of 3 is not needed for this case).

We next consider the case where $|s| > K$ and $|t| \leq K$. Here, we wish to bound the quantity $|s K^\alpha- t|t|^{\alpha}|$ by the quantity $3|s-t|(K^\alpha + |t|^\alpha)$. We suppose that $s \geq K \geq t \geq 0$. Then, it suffices to show that
\[ \frac{sK^\alpha - t^{1+ \alpha}}{s-t} \leq 3 K^\alpha,\]
which holds if and only if
\[ \frac{s - t^{1+\alpha}K^{-\alpha}}{s-t} \leq 3.\]

We define $0 \leq c \leq 1$ by $t = cK$ and $\sigma>0$ by $s = (1+\sigma)K$. Then we have:
\[ \frac{s - t^{1+\alpha}K^{-\alpha}}{s-t} = \frac{ 1+\sigma - c^{\alpha+1}}{1+ \sigma - c}.\]
We observe:
\[\frac{ 1+\sigma - c^{\alpha+1}}{1+ \sigma - c} = \frac{1 - c^{1+\alpha}}{1 + \sigma - c} + \frac{\sigma}{1+\sigma-c} \leq \frac{1-c^{1+\alpha}}{1 - c} + \frac{\sigma}{\sigma}.\]
Since $\alpha \leq 1$ and $c \leq 1$, we note that $c^{1+\alpha} \geq c^2$, so this is $\leq 1 + c+ 1 \leq 3$, as required.

We now suppose instead that $s \geq K$, and $-K \leq t <0$. In this case, we need to show that
\[ \frac{sK^\alpha+|t|^{1+\alpha}}{s + |t|}\leq 3 K^\alpha.\]
Dividing out $K^\alpha$, we obtain the equivalent requirement
\[ \frac{s + |t|^{1+ \alpha}K^{-\alpha}}{s+|t|} \leq 3.\]
Using that $|t|\leq K\leq s$, we observe that $\frac{s + |t|^{1+ \alpha}K^{-\alpha}}{s+|t|}\leq \frac{2s}{s} = 2$, and so the desired inequality holds. The cases where $s < -K$ and $|t| \leq K$ can be handled symmetrically.

Finally, we must consider the case where $|s|, |t| \geq K$. In this case, we wish to bound the quantity
\[|s K^\alpha - t K^\alpha| = K^\alpha |s-t| = \frac{1}{2} (\gamma_K(s) + \gamma_K(t))|s-t|,\]
so this is clearly $\leq 3 |\gamma_K(s) + \gamma_K(t)|\cdot|s-t|$ as required.
\end{proof}

\begin{lemma}\label{lem:gammain} For $K\geq 1$, for all $t$ we have that
\[\Gamma_{K}(t) \leq t^{p}\]
\[\Gamma_{K}(t) \leq \left(1 + \frac{p-2}{2}\right)K^{p-2}t^2.\]
It follows that for $f: \mathbb{T} \rightarrow \mathbb{R}$ we have $||f||_{\Gamma_{K}} \leq ||f||_{p}$ and $||f||_{\Gamma_K} \leq \left(1 + \frac{p-2}{2}\right)^{1/2}K^{(p-2)/2}||f||_{2}$.
\end{lemma}

\begin{proof}The inequality $\Gamma_{K}(t) \leq \left(1 + \frac{p-2}{2}\right)K^{p-2}t^2$ is clear from the definition of $\Gamma_K(t)$. We prove $\Gamma_{K}(t) \leq t^{p}$. This is clear for $|t| \leq K$, so we assume $|t| > K$. We let $0\leq c \leq 1$ be defined by $c:= \frac{K}{t} $. It then suffices to show
\[ \left(1 + \frac{p-2}{2}\right) c^{p-2} - \left(\frac{p-2}{2}\right)c^p  \leq 1. \]

Setting the derivative (with respect to $c$) equal to $0$, we see that
\[0 = \frac{d}{dc} \left(1 + \frac{p-2}{2}\right) c^{p-2} - \left(\frac{p-2}{2}\right)c^p
= (p-2) \left(1+ \frac{p-2}{2}\right) c^{p-3} - \left(\frac{p-2}{2}\right)p c^{p-1}. \]
This implies that $(1+\frac{p-2}{2}) c^{p-3}= \frac{p}{2} c^{p-1}$ and it follows that $c=1$. Lastly the inequality can be easily verified at $c=0,1$.

Let $f$ be a function from $\mathbb{T}$ to $\mathbb{R}$, and let $\lambda = ||f||_{p}$. Now, using $\Gamma_{K}(t) \leq t^{p}$,
\[ \int \Gamma_{K}(f/\lambda) \leq ||f||_{p}^{-p} \int |f|^{p} \leq 1.\]
Similarly, setting $\lambda = \left(1 + \frac{p-2}{2}\right)^{1/2}K^{(p-2)/2}||f||_2$ and using $\Gamma_{K}(t) \leq \left(1 + \frac{p-2}{2}\right)K^{p-2}t^2$, we have
\[\int \Gamma_{K}(f/\lambda) \leq \left(\left(1 + \frac{p-2}{2}\right)^{1/2}K^{(p-2)/2}||f||_{2}\right)^{-2} \int  \left(1 + \frac{p-2}{2}\right)K^{p-2}|f|^2 \leq 1. \]
\end{proof}

Given $\Gamma_K$, we also define its \emph{dual}, $\Gamma^*_K: \mathbb{R} \rightarrow \mathbb{R}$, as
\[\Gamma^*_K(x) = \int_0^{|x|} \left(\Gamma'_K\right)^{-1}(t) dt.\]
By a straightforward computation, we have for $t \geq 0$:
\[ \left(\Gamma'_K\right)^{-1} (t) = \left\{
                                       \begin{array}{ll}
                                         \left( t/p\right)^{\frac{1}{p-1}}, & \hbox{if $t \leq p K^{p-1}$;} \\
                                         t/\left(p K^{p-2}\right), & \hbox{if $t \geq p K^{p-1}$.}
                                       \end{array}
                                     \right.\]
We then compute that for $0 \leq x \leq p K^{p-1}$
\begin{equation}\label{eq:dualLow}
\Gamma^*_K (x) = \int_0^x \left(\frac{t}{p}\right)^{\frac{1}{p-1}} dt = p^{-\frac{p}{p-1}} (p-1) x^{\frac{p}{p-1}}.
\end{equation}
For $x > p K^{p-1}$, we compute
\begin{equation}\label{eq:dualHigh}
\Gamma^*_K(x) = \int_0^{pK^{p-1}} \left(\frac{t}{p}\right)^{\frac{1}{p-1}} dt + \int_{pK^{p-1}}^x \frac{t}{pK^{p-2}} dt = \frac{x^2}{2pK^{p-2}} + \left( \frac{p-2}{2}\right) K^p.
\end{equation}

We call $\Gamma^*_K$ the dual of $\Gamma_K$ because $||\cdot ||_{\Gamma^*_K}$ is the equivalent to the dual norm of $||\cdot ||_{\Gamma_K}$. More precisely:

\begin{lemma}\label{lem:dual}
There exist (universal) positive constants $C_1, C_2$ such that, for all $f$,
\[ C_1 \sup_{||g||_{\Gamma^*_K} \leq 1} \int fg \leq ||f||_{\Gamma_K} \leq C_2 \sup_{||g||_{\Gamma^*_K}\leq 1} \int fg.\]
\end{lemma}

This follows from the standard theory of Orlicz spaces (see Chapter 2 of \cite{KR}, for instance).

\begin{lemma}\label{lem:pieces} For any measurable $f: \mathbb{T} \rightarrow \mathbb{R}$, we can decompose $f=f_1 + f_2$ such that
\[||f_1 ||_{L^p} \ll ||f||_{\Gamma_{K}} \text{ and } ||f_2||_{L^2} \ll K^{(2-p)/2} ||f||_{\Gamma_K} .  \]
\end{lemma}

\begin{proof}By homogeneity we may assume that $||f||_{\Gamma_{K}}=1$. We let $\mathbb{I}_S$ denote the indicator function of a set $S \subset \mathbb{T}$. We now define
\[f_1 := f \cdot \mathbb{I}_{\{ x : |\frac{f(x)}{2}| <  K\}} \text{ and } f_2 := f \cdot \mathbb{I}_{\{x: |\frac{f(x)}{2}| \geq  K\}}.  \]

Using the hypothesis that $||f||_{\Gamma_{K}}=1$, we have
\[\int \Gamma_{K}(f/2) =  \int (f/2)^{p}  \mathbb{I}_{\{x:|\frac{f(x)}{2}| <  K\}}  + \int \left( \left(1 + \frac{p-2}{2}\right)K^{p-2}(f/2)^2 - \left(\frac{p-2}{2}\right) K^p \right) \mathbb{I}_{\{x:|\frac{f(x)}{2}| \geq  K\}}     \leq 1. \]
It follows that
\[ \int (f/2)^{p}  \mathbb{I}_{\{x:|\frac{f(x)}{2}| <  K\}}  \leq 1, \]
which is equivalent to $||f_1||_{p} \leq 2$ (or $||f_1||_{p} \leq 2||f||_{\Gamma_K}$). Next we have that
\[\int \left( \left(1 + \frac{p-2}{2}\right)K^{p-2}(f/2)^2 - \left(\frac{p-2}{2}\right) K^p \right) \mathbb{I}_{\{x:|\frac{f(x)}{2}| \geq  K\}}\]
\[=\int \left( 1 + \frac{p-2}{2}\right)K^{p-2}(f/2)^2  \mathbb{I}_{\{x:|\frac{f(x)}{2}| \geq  K\}} - \mu\left(\left\{x:\left|\frac{f(x)}{2}\right| \geq  K \right\}\right) \frac{p-2}{2} K^p \leq 1, \]
where $\mu$ denotes the Lebesgue measure.

Since we are assuming $||f||_{\Gamma_K} = 1$ and $\Gamma_K(f(x)/2) \geq K^p$ whenever $|f(x)/2| \geq K$, we must have
$\mu\left( \left\{ x: \left|\frac{f(x)}{2}\right|\geq K\right\}\right)\leq K^{-p}$.
Combining this with the above inequality, we see that
\[ \int \left( 1 + \frac{p-2}{2}\right)K^{p-2}(f/2)^2  \mathbb{I}_{\{x:|\frac{f(x)}{2}| \geq  K\}}  \leq 1 + \frac{p-2}{2}.\]
This implies
\[ \int (f_2)^2   \leq 4 K^{2-p}.\]

\end{proof}

\begin{lemma}\label{lem:supportBound}We have that $||f||_{2} \ll   p K^{p-1}||f||_{\Gamma_{K}^{*}}$.
\end{lemma}

\begin{proof}
Referring to the computation of $\Gamma_{K}^*$ in (\ref{eq:dualHigh}), we see that \[\Gamma_{K}^*(x) \geq \mathbb{I}_{\geq p K^{p-1}}(x) \cdot \frac{x^2}{2pK^{p-2}}.\]
Let $||f||_{\Gamma^*}= 1$, then
\begin{equation}\label{eq:Mbound1}
\frac{1}{2pK^{p-2}}\int_{\mathbb{T}} \mathbb{I}_{\geq p K^{p-1}} (f) f^2 \leq \int_{\mathbb{T}} \Gamma_{K}^*(f) \leq 1.
\end{equation}
We also note that
\begin{equation}\label{eq:Mbound2}
||f||_{L^2}^2  = \int_{\mathbb{T}} \mathbb{I}_{ \leq p K^{p-1}} (f) f^2 + \int_{\mathbb{T}} \mathbb{I}_{\geq p K^{p-1}}(f) f^2 \leq p^2 K^{2p-2} + \int_{\mathbb{T}} \mathbb{I}_{\geq p K^{p-1}}(f) f^2.
\end{equation}
Combining (\ref{eq:Mbound1}) and (\ref{eq:Mbound2}), we have $||f||_{L^2}^2 \ll p^2 K^{2p-2}$.

\end{proof}

We will now recall several probabilistic results that we will need later.  The following lemma (see p. 229 of \cite{BourLp}) is the chaining argument from Bourgain's work on the $\Lambda(p)$-problem:
\begin{lemma}\label{lem:chaining} Let $\mathcal{E} \subset \mathbb{R}^{N}$ and $B := \sup_{x \in \mathcal{E}}||x||_{\ell^2}$ (the diameter of $\mathcal{E}$).  Let $0 <\delta <1$ and $\{\xi_i\}_{i=1}^{N}$ independent $0,1$-valued random variables (selectors) with mean $\delta = \int \xi_i(\omega) d \omega$ and $1 \leq m \leq N$. Then for any $q \geq 1$,
\begin{equation}
\left| \left| \sup_{x \in \mathcal{E}, |S| \leq m} \sum_{i \in S} \xi_i(\omega) x_i  \right| \right|_{L^{q}(d\omega)}
\ll \left[\delta m + \frac{q}{\log(1/\delta)} \right]^{1/2} + \log^{-1/2}(1/\delta) \int_{0}^{B} \log^{1/2} (N_{2}(\mathcal{E},t)) dt.
\end{equation}
\end{lemma}
Here, $N_2(\mathcal{E},t)$ denotes the entropy number of $\mathcal{E}$ with respect to the $\ell^2$-distance. In other words, $N_2(\mathcal{E},t)$ is the minimum number of $\ell^2$-balls of radius $t$ needed to cover the set $\mathcal{E}$.

The following lemma (see p. 231 of \cite{BourLp}) is the entropy estimate from Bourgain's work on the $\Lambda(p)$-problem. There it is stated for orthonormal systems uniformly bounded by 1, but it generalizes easily to those uniformly bounded by a fixed constant $A$:
\begin{lemma}\label{lem:entropy} Let $\Phi:=\{\phi_i\}_{i=1}^{N}$ denote an orthonormal system of functions uniformly bounded by $A$. Further let $1\leq m\leq N$ and $2\leq q < \infty$ and define
\[\mathcal{P}_m := \left\{ \sum_{i \in S} a_i \phi_i : \sum a_i^2 \leq 1 , |S| \leq m \right\}.\]
Then, for some $v_q >2$,
\begin{equation}
\log \left( N_q(\mathcal{P}_{m},t) \right) \ll_A C_q m \log\left(\frac{N}{m}+1\right) t^{-v_q} \text{ for } t>\frac{1}{2}
\end{equation}
\begin{equation}
\log \left( N_q(\mathcal{P}_{m},t )\right) \ll_A C_q m \log\left(\frac{N}{m}+1\right) \log(1/t) \text{ for } 0<t \leq \frac{1}{2}.
\end{equation}
\end{lemma}

The following lemma appears as Theorem 5.2 in \cite{TalagrandSmooth} (see also \cite{TalagrandSelect} and \cite{TalagrandGeneric}). This deep and general result is a key technical component of Talagrand's alternate approach to the $\Lambda(p)$-problem and is closely related to what is often called the majorizing measures theorem.
\begin{lemma}\label{lem:Talagrand} Consider an operator $U$ from $\ell_{N}^{2}$ to the Banach space of real valued functions on $\mathbb{T}$ with a norm $||\cdot||$. Further assume that $||\cdot|| \leq || \cdot ||_{p}$ for some $p>2$. Now consider $N$ independent mean $\delta$ selectors $\{\xi_i\}_{i=1}^{N}$ and consider the random subset $S := \{i \in [N] : \xi_i = 1 \}$. Then, if we denote by $U|_{S}$ the restriction of domain of the operator $U$ to sequences supported on $S$, we have that
\begin{equation}\label{equ:Tala}
 \mathbb{E}\left[ || U|_{S} ||^2 \right]\ll \frac{||U||^2 + C_p \log(N) }{ \log(1/\delta) }.
\end{equation}
\end{lemma}

The following lemma appears as Theorem 7.4 in \cite{Ledoux} and provides a strong concentration bound for an empirical processes. This is also due to Talagrand \cite{TalagrandProd}, although from work on a somewhat different topic.
\begin{lemma}\label{lem:Ledoux} Let $Y_1,\ldots, Y_N$ denote random variables taking values in a Banach space $W$, and $\mathcal{F}$ a countable collection of measurable functions on $W$ pointwise bounded by $C$ (i.e. $|f| \leq C$). Set $Z := \sup_{f \in \mathcal{F}} \sum_{i=1}^{N} f(Y_i)$ and $\sigma^{2}=\sup_{f \in \mathcal{F}} \sum_{i=1}^{N} |f(Y_i)|^2$. Then, for all $\tau>0$,
\begin{equation}
\mathbb{P} \left( |Z - \mathbb{E}(Z)| \geq \tau \right) \leq 3 \exp\left(-\frac{\tau}{\kappa C} \log\left(1+ \frac{C\tau}{\mathbb{E}\sigma^2}\right)\right).
\end{equation}
for some absolute constant $\kappa$.
\end{lemma}

We will also need the following standard fact:

\begin{lemma}\label{lem:french}
Let $(V_y)_{y \in E}$ be a set of real, non-negative random variables indexed by a finite set $E$ of size $N$. Let $q \geq 1$. Then
\[ \left|\left| \sup_y V_y \right|\right|_{q} \ll \sup_y ||V_y||_{q + \log N}.\]
\end{lemma}

\begin{proof}
We define $Z:= \sup_y V_y^{q}$ and let $\eta \geq 1$ be a parameter to be set later. Applying Jensen's inequality, we have $\left( \mathbb{E}[Z]\right)^{\eta} \leq \mathbb{E}[Z^\eta]$. We note that $Z^\eta \leq \sum_y V_y^{q \eta}$, so by the linearity of expectation we have
\[\left( \mathbb{E}[Z]\right)^{\eta} \leq \sum_y \mathbb{E}[V_y^{q \eta}] \leq N \sup_y ||V_y||_{q \eta}^{q\eta}.\]
We then observe:
\[\left|\left| \sup_y V_y \right|\right|_{q} = \left( \mathbb{E}[Z]\right)^{\frac{1}{q}} \leq N^{\frac{1}{q \eta}} \sup_y ||V_y||_{q \eta}.\]
We then choose $\eta$ so that $\eta q = q + \log N$. We then have $N^{\frac{1}{q \eta}} \leq N^{\frac{1}{\log N}} \ll 1$, so the lemma follows.
\end{proof}

Finally, we recall Pittel's inequality, which can be found on p.38 in \cite{Bo}, for example:
\begin{lemma}\label{lem:SelToSet} For positive integers $ 1 \leq m \leq N$, we let $[N] \choose m$ denote the set of all subsets of $[N]$ of size $m$. We let $S_m$ denote a uniformly random element of $[N] \choose m$. We let $\widetilde{S}$ denote a subset of $[N]$ chosen by including each element of $[N]$ independently with probability $\delta = m/N$. Then, for any event $E$ described as a set of subsets of $[N]$, we have
\[ \mathbb{P}\left[ S_m \in E \cap {[N] \choose m}\right] \ll \sqrt{m}\cdot \mathbb{P}\left[ \widetilde{S} \in E\right].\]
\end{lemma}

\section{Dyadic Decompositions}

Let $\Phi = \{\phi_i \}_{i=1}^{N}$ be an orthonormal system and fix $f= \sum_{n \in [N]} a_n \phi_n$ in the span of the system. We now perform a dyadic decomposition of $f$ in terms of the $\ell^2$ weight of the coefficients (which we refer to as `mass'). We reproduce the description of this decomposition from \cite{LewkoO}.

We  define the \emph{mass} of a subinterval $I \subseteq [N]$ as $M(I) := \sum_{n \in I} |a_{n}|^2$. By normalization, we may assume that $M([N])=1$. We define $I_{0,1} := [N]$ and we iteratively define $I_{k,s}$ for $1\leq s\leq 2^k$ as follows. Assuming we have already defined $I_{k-1,s}$ for all $1 \leq s \leq 2^{k-1}$, we will define $I_{k,2s-1}$ and $I_{k,2s}$, which are subintervals of $I_{k-1,s}$. $I_{k,2s-1}$ begins at the left endpoint of $I_{k-1,s}$ and extends to the right as far as possible while covering strictly less than half the mass of $I_{k-1,s}$, while $I_{k,2s}$ ends at the right endpoint of $I_{k-1,s}$ and extends to the left as far as possible while covering at most half the mass of $I_{k-1,s}$. More formally, we define $I_{k,2s-1}$ as the maximal subinterval of $I_{k-1,s}$ which contains the left endpoint of $I_{k-1,s}$ and satisfies $M(I_{k,2s-1}) < \frac{1}{2} M(I_{k-1,s})$. We also define $I_{k,2s}$ as the maximal subinterval of $I_{k-1,s}$ which contains the right endpoint of $I_{k-1,s}$ and satisfies $M(I_{k,2s}) \leq \frac{1}{2} M(I_{k-1,s})$. We note that these subintervals are disjoint. We may express $I_{k-1,s} = I_{k,2s-1} \bigcup I_{k,2s} \bigcup i_{k,s}$,  where $i_{k,s} \in I_{k-1,s}$. In other words, $i_{k,s}$ denotes the single element which lies between $I_{k,2s-1}$ and $I_{k,2s}$ (note that such a point always exists because we have required that $I_{k,2s-1}$ contains strictly less than half of the mass of the interval). Here it is acceptable for some choices of the intervals in this decomposition to be empty. By construction we have that

\begin{equation}\label{eq:mass}
M(I_{k,s}) \leq 2^{-k}.
\end{equation}

For $J \subseteq [N]$, we define
\[f_J(x) = \sum_{n\in J} a_n \phi_n(x).\]
We also define
\[ \tilde{f}_J(x) := \max_{I \subseteq J} \left|\sum_{n \in I} a_n \phi_n(x)\right| .\]
where $\max_{I \subseteq J}$ denotes the maximum over subintervals.

\begin{lemma}\label{lem:maxde}Let $\Phi = \{\phi_i \}_{i=1}^{N}$  denote an orthonormal system such that $|\phi_i(x)| \leq A$ for all $i,x$ and for every $f=\sum_{n \in [N]} a_n \phi_n$ in the span of the system one has
$$ ||f ||_{\Gamma_K} \leq \Delta ||f||_{2}$$
for some constant $\Delta \geq 2$. Then we may decompose the maximal function $\tilde{f} = f_1 +f_2$ where
$$||f_1||_{p} \ll_{p} \Delta ||f||_{2}  $$
$$||f_2||_{2} \ll_{p} \Delta \log(N) K^{(2-p)/2}||f||_{2}. $$
\end{lemma}

\begin{proof} We normalize $||f||_{2} =1$ and split $f=f_{L}+f_{S}$ where $f_{L}:= \sum_{ |a_n| > N^{-1}} a_n \phi_n  $ and $f_{S}:= \sum_{ |a_n| \leq N^{-1}} a_n \phi_n $. Now $\tilde{f} \leq \tilde{f_L}+\tilde{f_S}$ and $\tilde{f_S} \leq \sum_{|a_n| \leq N^{-1}} |a_n| |\phi_n| \leq A $. Clearly $|| \tilde{f_S}||_{p} \ll ||f||_2$, so it remains to prove the desired properties for $\tilde{f}_L$.

We next perform the `mass' decomposition described above on $f_L$. On each interval $I_{k,s}$ in the dyadic decomposition of $f_L$, we can apply Lemma \ref{lem:pieces} and the hypothesis $||f ||_{\Gamma_K} \leq \Delta ||f||_{2}$ to express the restriction of $f_L$ to $I_{k,s}$ as the sum of two functions $G_{k,s}$ and $E_{k,s}$ such that $||G_{k,s}||_p \ll \Delta ||f_{I_{k,s}}||_2$ and $||E_{k,s}||_2 \leq \Delta K^{(2-p)/2} ||f_{I_{k,s}}||_2$. Now, at each point $x$, the maximizing subinterval achieving the value $\tilde{f}_L(x)$ can be decomposed into a disjoint union of dyadic intervals (and points) using at most two intervals on each level.
We then have the pointwise inequality
\[ \tilde{f}_{L} \ll \sum_{k=1}^{\lceil \log(N)\rceil} \left(\sum_{s} |G_{k,s} + E_{k,s} |^p \right)^{1/p}  + \sum_{k} \left(\sum_{s} | f_{i_{k,s}} |^p \right)^{1/p}\]
\[\ll \sum_{k=1}^{\lceil \log(N)\rceil} \left(\sum_{s} |G_{k,s}  |^p \right)^{1/p} + \sum_{k=1}^{\lceil \log(N) \rceil} \left(\sum_{s} | E_{k,s} |^p \right)^{1/p} + \sum_{k} \left(\sum_{s} | f_{i_{k,s}} |^p \right)^{1/p} \]
where we have used the condition $|a_n| \geq N^{-1}$ to restrict the sum in $k$ to $\lceil \log(N) \rceil$ values. (We have also used that taking the $\ell_p$-norm of the values of the intervals on a single level provides an upper bound on the highest value taken on that level.)

Note that it suffices to prove that an appropriate decomposition exists for a pointwise majorant.
We treat each of the three sums on right of the above inequality separately. By hypothesis, we have that $||G_{k,s}||_p \leq \Delta ||f_{I_{k,s}}||_{2}$ and $||E_{k,s}||_{2} \leq  \Delta K^{(2-p)/2}||f_{I_{k,s}}||_{2}$.  We then see that
\[ \left|\left|\sum_{k=1}^{\lceil\log(N)\rceil} \left(\sum_{s} |G_{k,s}  |^p \right)^{1/p}\right|\right|_{p}
\leq \sum_{k=1}^{\lceil\log(N)\rceil} \left|\left|\left(\sum_{s} |G_{k,s}  |^p \right)^{1/p}\right|\right|_{p}   \]
\[= \sum_{k=1}^{\lceil\log(N)\rceil} \left(\sum_{s}  || G_{k,s}||_{p}^{p} \right)^{1/p}
\leq  \sum_{k=1}^{\lceil\log(N)\rceil }\Delta 2^{k(1/p - 1/2)} \ll_{p} \Delta.\]
Here we have used that $||f_{I_{k,s}}||_{2} \leq 2^{-k/2}$.

Next we have
\[\left|\left| \sum_{k=1}^{\lceil \log(N)\rceil} \left(\sum_{s} | E_{k,s} |^p \right)^{1/p} \right|\right|_{2}
\leq \left|\left| \sum_{k=1}^{\lceil \log(N) \rceil} \left(\sum_{s} | E_{k,s} |^2 \right)^{1/2} \right|\right|_{2} \]
\[ \leq \sum_{k=1}^{\lceil \log(N)\rceil}  \left( \sum_{s} ||   E_{k,s} ||_{2}^{2} \right)^{1/2} \ll  \Delta  \log(N) K^{(2-p)/2}.\]
Recall that $f_{i_{k,s}}= a_j \phi_j$ for some $j \in [N]$. By hypothesis we have $|\phi_j(x)| \leq A $ for all $x$, furthermore  $|a_j| \leq 2^{-(k-1)/2}$ since $j \in I_{k-1,s'}$ for some $s'$. Thus $|f_{i_{k,s}}| \ll_A 2^{-(k-1)/2}$. Hence, for any $x$,
\[\sum_k \left(\sum_s |f_{i_{k,s}}(x)|^p\right)^{\frac{1}{p}} \ll \sum_{k} \left(\sum_{s} 2^{-(p-2)(k-1)/2} | f_{i_{k,s}}(x) |^2 \right)^{1/p}
\ll \sum_{k}  2^{-(p-2)(k-1)/(2p)} \ll_{p} 1.\]
Since this term is pointwise bounded, clearly its $L^p$ norm is also bounded. Since the decomposition holds for each of the three terms it clearly holds for their sum, completing the proof.
\end{proof}

We now record the following easily verified fact (see Lemma 29 in \cite{LewkoO}).

\begin{lemma}\label{lem:decomposition}We fix $||f||_{2}=1$ (in the span of $\{\phi_i\}_{i=1}^{N}$) and let  $\mathcal{A}$ denote the set of intervals that occur in the mass decomposition of $f$.  For every interval $J \subseteq [N]$, ($J \neq \emptyset$), there exist $\tilde{J}_\ell, \tilde{J}_r \in \mathcal{A}$ and $i_J \in [N]$ such that $\tilde{J}:= \tilde{J}_{\ell} \cup i_J \cup \tilde{J}_r$ is an interval (i.e. $\tilde{J}_\ell, i_J, \tilde{J}_\ell$ are adjacent), $J \subseteq \tilde{J}$, and $M(\tilde{J}) \leq 2M(J)$.
\end{lemma}

\begin{lemma}\label{lem:v2subinterval}Let $\Phi = \{\phi_i \}_{i=1}^{N}$ denote a bounded orthonormal system. Furthermore, let $M < N$, let $C', r >0$ be positive constants, and let $K \geq 2$. We assume that for any interval $I \subseteq [N]$ of length $|I| \leq M$ we have for any $h=\sum_{n \in I} a_n \phi_n$, the estimate
$$ ||h||_{\Gamma_K} \leq \frac{C' \log^{r}(N)} {\log^{r}(N/|I|)} ||h||_{2}$$
holds. In addition, we assume that
$$||\mathcal{M}f||_{2} \ll \log\log(N) ||f||_{2}  $$
holds for any $f$ in the span of the system.  Then, if $f=\sum_{n \in \mathcal{I}} a_n \phi_n$ for $|\mathcal{I}| =  M \leq N$, for any $\beta >0$ we have
\begin{eqnarray}\label{eq:v2sub}
\nonumber
  ||\mathcal{V}^2f ||_{2} & \ll_{p}& \big( \log^{\beta/2}(N) +(C')^{p/2}\log^{rp/2-\beta (p-2)/4}(N) \log^{(1-pr)/2}(N/M)  \\
   & &  + \; C'  \log^{r+1} (N) K^{(2-p)/2}  +\log \log (N)   \big) ||f||_{2}.
\end{eqnarray}
\end{lemma}

\begin{proof}Let $f = \sum_{n \in \mathcal{I} } a_n \phi_n$ with $|\mathcal{I}| = M$ and  $\mathcal{I} \subseteq [N]$.  Normalize $||f||_{2} =1$, and again split $f=f_{L}+f_{S}$ where $f_{L}= \sum_{ |a_n| > N^{-1}} a_n \phi_n  $ and $f_{S}= \sum_{ |a_n| \leq N^{-1}} a_n \phi_n $. We then have $\mathcal{V}^2 f_{S} \leq \sum |a_n \phi_n| \ll 1$ and may restrict our attention to $f_L$.

We now perform a mass decomposition on $f_{L}$, using the notation above. Let $I$ be a dyadic subinterval. We consider the restriction of $f_L$ to this interval, and we apply Lemma \ref{lem:maxde} to decompose the maximal version of this restriction as a sum of functions $\tilde{G}_{I} + \tilde{E}_{I}$. We may apply Lemma \ref{lem:maxde} to the system $\{\phi_i \}_{i \in I}$ with $\Delta := \frac{C'\log^r(N)}{\log^r(N/|I|)}$. We then obtain
\[ ||\tilde{G}_{I}||_p \ll_p \frac{C'\log^r(N)}{\log^r(N/|I|)} ||f_{I}||_2,\]
\[ ||\tilde{E}_{I} ||_2 \ll_p \frac{C'\log^{r+1}(N)}{\log^r(N/|I|)} K^{(2-p)/2} ||f_{I}||_2.\]

For a fixed $x \in \mathbb{T}$, we know that the value of $\mathcal{V}^2f_{L} (x)$ is achieved by some maximizing partition $\pi$.
Appealing to Lemma \ref{lem:decomposition}, we know that each interval in $\pi$ can be covered by dyadic intervals and points (and each dyadic interval/point will only be used to cover at most a constant number of the original intervals). Let $\mathcal{A}$ denote the family of dyadic intervals and points. We write $\pi^* \subset \mathcal{A}$ to denote a suitable covering of a partition by dyadic intervals and points.
For each $k$, we define $I^a_k := \{I_{k,s} \text{ s.t. } |I_{k,s}| \leq 2^{-k/2}M\}$ and $I_k^b := \{I_{k,s} \text{ s.t. } |I_{k,s}| > 2^{-k/2}M\}$. We let $I^a$ denote the union of $I^a_k$ over the values of $k$ let $I^b$ denote the union of $I^b_k$.
We then have
\[ |\mathcal{V}^2f_{L} (x)|^2 \ll   \sup_{\pi } \sum_{I \in \pi} | f_{I}(x)  |^2 \ll \sup_{\pi^* \subset \mathcal{A} }\left( \sum_{I \in \pi^*} | \tilde{f}_{I}  |^2 + \sum_{i_{k,s} \in \pi^*} |f_{i_{k,s}}|^2  \right)\]
\[ \ll \sup_{\pi^* \subset \mathcal{A} } \sum_{\substack{I \in \pi^* \\ I \in I^{a} }} | \tilde{f}_{I}  |^2 + \sup_{\pi^* \subset \mathcal{A}} \sum_{\substack{I \in \pi^* \\ I \in I^{b} }} | \tilde{f}_{I}  |^2 + \sup_{\pi^* \subset \mathcal{A}} \sum_{i_{k,s} \in \pi^*} |f_{i_{k,s}}|^2. \]

The third quantity here can be easily handled. Since $f_{i_{k,s}} = a_{i_{k,s}} \phi_{i_{k,s}}$ and $\sum_i a_i^2 = 1$, the contribution from the points here is $\ll 1$.
The second sum may also be estimated in a crude manner. Recall that $|| \tilde{f}_{I_{k,s}} ||_{2}^2 \ll (\log \log (N))^2 2^{-k} $ (by assumption) and note that $|I_{k}^{b}| \leq 2^{k/2}$. Thus
\[\int \sup_{\pi^* \subset \mathcal{A} } \sum_{\substack{I \in \pi^* \\ I \in I^{b} }} | \tilde{f}_{I}  |^2 \ll  \sum_{k} \sum_{\substack{I_{k,s} \in I_{k}^{b}}  } (\log \log (N))^2  2^{-k} \ll  (\log \log(N))^2 \sum_{k} 2^{-k/2} \ll (\log \log (N))^2. \]

We return to the first sum.
We fix $\beta>0$.
For each dyadic interval $I$, we define the `bad' set  $B_{I} = \{x \in \mathbb{T} : |\tilde{G}_{I}(x) |^2 \geq \log^{\beta}(N) M(I)\}$ and we let $H_I$ denote its complement inside $\mathbb{T}$. Then,
\[\int \sup_{\pi^* \subset \mathcal{A} } \sum_{\substack{I \in \pi^* \\ I \in I^{a} }} | \tilde{f}_{I}  |^2 \ll \int \sup_{\pi^*\subset \mathcal{A} } \sum_{\substack{I \in \pi^* \\ I \in I^{a} }} | \tilde{G}_{I} + \tilde{E}_{I}  |^2\]
 \[\ll  \int   \sup_{\pi^* \subset \mathcal{A}} \sum_{\substack{I \in \pi^* \\ I \in I^{a} }} | \mathbb{I}_{H_{I}} \tilde{G}_{I}  |^2   + \int \sup_{\pi^* \subset \mathcal{A}} \sum_{\substack{I \in \pi^* \\ I \in I^{a} }}   | \mathbb{I}_{B_{I}}  \tilde{G}_{I} |^2  +  \int \sup_{\pi^* \subset \mathcal{A}}   \sum_{\substack{I \in \pi^* \\ I \in I^{a} }}   | \tilde{E}_{I} |^2 . \]

By Definition, $|\mathbb{I}_{H_{I}}  \tilde{G}_{I}(x) |^2 \leq \log^{\beta}(N) M(I)  $ for all $x$, and since the intervals in the partition $\pi$ are disjoint, we have that $\sum_{I \in \pi^*} M(I) \ll 1$. Thus
\[ \int \sup_{\pi^* \subset \mathcal{A} }  \sum_{\substack{I \in \pi^* \\ I \in I^{a} }}  | \mathbb{I}_{H_{I}}  \tilde{G}_{I}(x) |^2 \leq \log^{\beta}(N).\]
Next, using that $\int | E_{k,s}|^2   \ll_{p}  (C')^2 \log^{2(r+1)}(N) \log^{-2r}(N/|I_{k,s}|) K^{2-p} M(I_{k,s})$, we have
\[\int \sup_{\pi^* \subset \mathcal{A}}   \sum_{\substack{I \in \pi^* \\ I \in I^{a} }}   | \tilde{E}_{I} |^2    \ll_{p} (C')^2  \log^{2r+2} (N) K^{2-p}.  \]

It remains to estimate the contribution of the bad sets. By Chebyshev's inequality, we have
\[ \int \mathbb{I}_{B_{k,s}} \leq \frac{\int | \tilde{G}_{k,s} |^p  }{\log^{\beta p/2}(N) (M(I_{k,s}))^{p/2}  } \ll_{p} (C')^{p}\log^{pr}(N) \log^{-pr}(N/|I_{k,s}|) \log^{-\beta p/2}(N).\]
Now, by Holder's inequality, we have
\[ \int | \mathbb{I}_{B_{k,s}}  \tilde{G}_{k,s} |^2 \leq \left( \int \mathbb{I}_{B_{k,s}}   \right)^{1/(p/2)'}
\left( \int |\tilde{G}_{k,s} |^p  \right)^{2/p} \]
where $(p/2)'$ denotes the conjugate exponent of $p/2$. This is
\[\ll_{p } \left( (C')^{p}\log^{pr}(N) \log^{-pr}(N/|I_{k,s}|) \log^{-\beta p/2}(N)  \right)^{1/(p/2)'}   (C')^{2} \log^{2r}(N)\log^{-2r}(N/|I_{k,s}|)  M(I_{k,s})    \]
\[= (C')^{p} \log^{pr}(N)\log^{-pr}(N/|I_{k,s}|) \log^{-\beta (p-2)/2}(N)  M(I_{k,s}),  \]
using that $1/(p/2)' = (p-2)/p$. Since $|I_{k,s}| \leq 2^{-k/2}M$ for $I_{k,s} \in I^{a}$, we have that (assuming $pr >1$)
\[ \int \sum_{k,s, I_{k,s} \in I^a } | \mathbb{I}_{B_{k,s}} \tilde{G}_{k,s}  |^2  \ll_{p} (C')^{p}\log^{pr}(N) \log^{-\beta (p-2)/2}(N) \left( \sum_{k=1}^{\log(N) } (\log(N/M) + k)^{-pr} \right) \]
\[\ll_{p}   (C')^{p}\log^{pr}(N) \log^{-\beta (p-2)/2}(N) \int_{\log(N/M)}^{\infty} x^{-pr} dx \ll_{p}   (C')^{p}\log^{pr}(N) \log^{-\beta (p-2)/2}(N)  \log^{1-pr}(N/M). \]

\end{proof}

\begin{corollary}\label{cor:v2subinterval}
With the same hypothesis as Lemma \ref{lem:v2subinterval}, if $C'  \log^{r+1} (N) K^{(2-p)/2}  \ll 1 $,  $\log(N/M) \ll \log^{\theta}(N)$, and $r + \theta/p - \theta r >0$, then
$$||\mathcal{V}^2 f ||_{2} \ll_{p, C'}  \log^{r + \theta/p -\theta r} (N)||f||_{2}$$
for all $f = \sum_{n \in \mathcal{I}} a_n \phi_n$ with $|\mathcal{I}| = M$.
\end{corollary}
\begin{proof}
By hypothesis, the third term in (\ref{eq:v2sub}) is $\ll 1$, so we only need to choose a $\beta>0$ to balance the first two terms. Solving
$\beta/2 = rp/2  - \beta(p-2)/4 +\theta(1-pr)/2$ we see that $\beta = 2(r + \theta/p - \theta r)$. Taking this choice of $\beta$ in (\ref{eq:v2sub}) yields the corollary.
\end{proof}

At the beginning of this section we introduced a `mass' dyadic decomposition of $[N]$ with respect to a function $f=\sum_{n \in [N]} a_n \phi_n$. Now we recall the more common `length' dyadic decomposition. Without loss of generality, we assume $N = 2^{\ell}$ for some positive integer $\ell$ (if $N$ is not a power of $2$ we will simply round up to the nearest power of $2$). Consider the collection of dyadic subintervals of the form $\mathcal{I}_{k,s}= (s2^k, (s+1) 2^k]$ for each $0\leq k \leq \ell$, $0 \leq s \leq 2^{\ell - k} -1$. Note that the we have used a slightly different indexing convention here, compared with the mass decomposition.

\begin{lemma}\label{lem:lengthcover} Let $J \subseteq [N]$ be an arbitrary subinterval. Then we may decompose $J= J_{l} \cup J_{r}$ as a union of disjoint intervals $J_l, J_r$ where (i) at least one of $J_{l}$ and $J_{r}$ is non-empty, and (ii)  $J_{l} \subseteq \mathcal{I}_{k,s}$ for some $k,s$ where $|J_{l}| \geq \frac{1}{2}|\mathcal{I}_{k,s}|$  (assuming $J_{l}$ is non-empty), and the same holds for $J_r$.
\end{lemma}

\begin{proof}
We define $b \in J$ to be of the form $m2^k$ where $k$ is maximal. In other words, we consider all multiples of powers of 2 inside $J$, and we set $b$ to be one of the multiples of the highest power appearing. We set $J_{l}:= \{1, \ldots, b\} \cap J$ and $J_{r}:= \{b+1, \ldots, N\} \cap J$. We consider $|J_l|$ and we let $k_l$ be the minimal integer such that $2^{k_l} \geq |J_l|$. Then we must have $k_l \leq k$ (otherwise, a multiple of $2^{k+1}$ would appear in $J$, contradicting maximality of $k$). Thus, there is a dyadic interval of length $2^{k_l}$ that covers $J_l$ and has length $\leq 2|J_l|$. An analogous argument applies to $J_r$.
\end{proof}

As above, if $f(x)=\sum_{n \in [N]} a_n \phi_n(x)$ is fixed and $I \subseteq [N]$ is an interval we will write $f_{I} = \sum_{n \in I} a_n \phi_n$ and $\tilde{f}_{I} = \max_{J \subseteq{I}} |\sum_{n \in J} a_n \phi_n(x)|$ (where the maximum is over all subintervals of $I$).

\begin{lemma}\label{lem:intlenapprox}Let $\Phi = \{\phi_i \}_{i=1}^{N}$ denote a bounded orthonormal system such that
$$||\mathcal{M}f||_{2} \ll \log \log (N) ||f||_{2}$$
holds for all $f$ in the span of $\Phi$. Suppose we have $0 < \theta < 1$, $r>0$, and $p >2$ such that for each subinterval $\mathcal{I} \subseteq [N]$ of length $M =2^m$ where $M$ is the largest power of $2$ that is $\leq N 2^{- \log^{\theta}(N)}$,
$$ || \mathcal{V}^2 f ||_{2}  \ll \log^{r+ \theta/p - \theta r}(N) || f||_{2} $$
holds for any $f = \sum_{n \in I } a_n \phi_n$. Then, for any $f = \sum_{n \in [N]} a_n \phi_n$,  we have that
\[||\mathcal{V}^2 f ||_{2} \ll (\log^{\theta/2}(N) \log\log(N) + \log^{r+ \theta/p - \theta r}(N)) ||f||_{2}. \]
In particular, when $\theta= \frac{r}{1/2 + r -1/p}$, we obtain
\[||\mathcal{V}^2 f ||_{2} \ll \log^{\frac{r}{1 + 2r -2/p}}(N) \log\log(N) ||f||_{2}.\]
\end{lemma}

\begin{proof}
Setting $f(x)= \sum_{n \in [N]} a_n \phi_n(x)$, we may normalize $||f||_{2}^{2} =\sum_{n\in[N]} a_n^2 = 1$. We now wish to estimate the quantity
\begin{equation}\label{eq:intvar}
 \int \max_{\pi \in \mathcal{P}_{N}} \sum_{I \in \pi } \left| \sum_{n\in I} a_n\phi_n(x) \right|^2
\end{equation}
where $\mathcal{P}_{N}$ ranges over the set of all partitions of $[N]$. From the elementary inequality $(a+b)^2 \leq 3 (a^2 +b^2)$ we have that  $|\sum_{n \in J} a_n \phi_n(x)|^2 \ll |\sum_{n \in J_{l}} a_n \phi_n(x)|^2 + |\sum_{n \in J_{r}} a_n \phi_n(x)|^2 $ whenever $J = J_l \cup J_r$ for disjoint intervals $J_l, J_r$. Using Lemma \ref{lem:lengthcover}, it follows that we may restrict the set of partitions $\mathcal{P}_{N}$ in (\ref{eq:intvar}) to the subclass $\mathcal{P}_{N}^{*}$ of permutations where each interval in each partition is contained in a dyadic interval of at most twice its length. That is,
\[(\ref{eq:intvar}) \ll  \int \max_{\pi \in \mathcal{P}_{N}^{*}} \sum_{I \in \pi } \left| \sum_{n\in I} a_n\phi_n(x) \right|^2.\]

For a fixed $x$, let $\pi(x)$ denote the partition in $\mathcal{P}_{N}^{*}$ achieving the maximum in the above expression at $x$. We then have that the above quantity can be expressed as
\[\int \sum_{\substack{I \in \pi(x) \\ |I| \geq N 2^{- 2 \log^{\theta}(N)} }} \left| \sum_{n\in I} a_n\phi_n(x) \right|^2 + \int \sum_{\substack{I \in \pi(x) \\ |I| < N 2^{- 2 \log^{\theta}(N)} }} \left| \sum_{n\in I} a_n\phi_n(x) \right|^2 := B_1 + B_2\]

We now consider the contribution from the quantity $B_1$. Let $I \in \pi(x)$, since $\pi(x) \in \mathcal{P}_{N}^{*}$ there is a dyadic interval $J_{I}$, such that $|I| \leq |J_{I}| \leq 2 |I|$. Since the intervals in $\pi(x)$ are disjoint, and the associated dyadic interval $J_{I}$ has length at most $2|I|$, it follows that any particular dyadic interval is associated to at most $2$ intervals in $\pi(x)$. We recall $\ell =  \log (N)$. It follows that
\begin{equation}\label{eq:boundI}
B_1 \ll \sum_{0 \leq k\leq 2 \ell^{\theta}} \; \sum_{0 \leq s < 2^{\ell-k}} | \tilde{f}_{\mathcal{I}_{k,s}}(x)|^2 \ll \log ^{\theta}(N) (\log \log (N))^2,
\end{equation}
where we have used that $\int |\tilde{f}_{I}|^2 \ll \int |\mathcal{M} f_{I}|^2 \ll (\log \log (N))^2 \sum_{n \in I} a_n^2 $.

Next, we must estimate the quantity $B_2$. Let $M = 2^m$ denote the largest power of $2$ less than or equal to $N 2^{-  \log^{\theta}(N)}$. Thus  $\frac{1}{2}N 2^{-  \log^{\theta}(N)} < M \leq N 2^{-  \log^{\theta}(N)}$. Consider the partition of $N$ into (dyadic) subintervals of length $M$,  $\{\mathcal{I}_{m,s}  \}_{s=0}^{2^{\ell-m} -1}$.  It now follows from the hypothesis that on every interval in this partition $\mathcal{I}_{m,s}$,  we have $||\mathcal{V}^2f||_{2} \ll \log^{r+\theta/p - \theta r}(N)||f||_{2}$ for any $f = \sum_{n \in I_{m,s}} a_n \phi_n$.

Next we note that each of the intervals in the partitions in $B_2$ is contained in a dyadic interval of length at most $2 \times N 2^{-2\log^{\theta}(N)} \leq M$. By the nesting of dyadic intervals, it follows that each interval in the sum defining $B_2$ is strictly contained in an element of the partition $\{\mathcal{I}_{m,s}\}$. Thus one has
\begin{equation}\label{eq:boundII}
B_2 \ll \sum_{s} |\mathcal{V}^{2} f_{\mathcal{I}_{m,s}}|^2  \ll \log^{2(r+ \theta/p - \theta r) }(N) ||f||_2^2.
\end{equation}
Combining \ref{eq:boundI} and \ref{eq:boundII} (and taking square roots) we have that
\[||\mathcal{V}^2 f ||_{2} \ll (\log^{\theta/2}(N) \log \log (N) + \log^{r + \theta/p - \theta r}(N)) ||f||_{2}.\]

\end{proof}

\section{Probabilistic Estimates}

\subsection{A First Estimate}
We begin by establishing a close variant of Lemma 3.4 in \cite{Bour}, mainly following the proof in \cite{Bour} having taking more care with logarithmic factors.

\begin{lemma}\label{lem:Bbasic} Let $\phi_1, \ldots, \phi_N$ be orthonormal functions on a probability space $\mathbb{T}, \mu$ satisfying $|\phi_i| \leq A$ everywhere on $\mathbb{T}$ for all $i$. Let $0 < \delta < 1$ and let $\{\xi_i\}_{i=1}^N$ be independent, $\{0,1\}$-valued random variables of mean $\delta$ on a probability space $\Omega$. Let $1 < K < \infty$, $2 < p <3$, $1 \leq m \leq N$ and $q_0 \geq 1$.
Then there exists a
\[\lambda \ll_p \delta^{\frac{1}{4}} K^{\frac{p-2}{2}} + \left(\frac{q_0}{\log (\frac{1}{\delta})}\right)^{\frac{1}{4}} + \left(\frac{\log N}{\log ( \frac{1}{\delta})}\right)^{\frac{1}{2p}}\]
such that
\[ \left|\left| \sup_{\substack{|S| \leq m\\ |c_i|^2 < 2/m}} \int_\mathbb{T} \Gamma_K\left(\frac{\sum_{i \in S} c_i \xi_i(\omega) \phi_i}{\lambda}\right) \right|\right|_{L^{q_0} (d\omega)} \leq 1.\]
\end{lemma}

\begin{proof}
For a fixed $\lambda>0$ (to be set later) and any fixed $\omega \in \Omega$ and $W$ of size $\leq m$ and coefficients $c_i$ satisfying $|c_i|^2 < 2/m$, we observe that
\begin{equation}\label{Bstep1}
    \int_\mathbb{T} \Gamma_K \left(\frac{\sum_{i \in W} c_i \xi_i(\omega) \phi_i}{\lambda}\right) \ll \lambda^{-2}\cdot \frac{1}{\sqrt{m}} \sup_{g \in \mathcal{P}_m} \sum_{i \in W} \xi_i(\omega) \left|\langle\phi_i, g \gamma_K(\lambda^{-1} g)\rangle\right|,
\end{equation}
where \[\mathcal{P}_m := \left\{ g = \sum_{i \in S} a_i \phi_i\; \big| \;\sum_i |a_i|^2 \leq 1 \text{ and } |S| \leq m\right\}.\]

We define $\mathcal{E} \subseteq \mathbb{R}^N$ as
\[ \mathcal{E} := \left\{ \left(\left|\langle \phi_i, g \gamma_K(\lambda^{-1} g) \rangle \right|\right)_{i=1}^N \; \big| \; g \in \mathcal{P}_m\right\}.\]
We also define $B := \sum_{x \in \mathcal{E}} ||x||_{\ell^2}$.

Then, applying Lemma \ref{lem:chaining}, we have
\begin{eqnarray}\label{Bstep2}\nonumber
\left|\left| \sup_{g \in \mathcal{P}_m, |W| \leq m} \sum_{i \in W} \xi_i (\omega) \cdot | \langle \phi_i, g \gamma_K(\lambda^{-1} g)\rangle | \right|\right|_{L^{q_0} (d\omega)} \\
\ll \left( \delta m + \frac{q_0}{\log(\frac{1}{\delta})}\right)^{\frac{1}{2}} B + \left(\log \left(\frac{1}{\delta}\right)\right)^{-\frac{1}{2}} \int_0^B \left( \log N_2(\mathcal{E},t)\right)^{\frac{1}{2}} dt,
\end{eqnarray}
where $N_2(\mathcal{E},t)$ denotes the entropy number of $\mathcal{E}$ with respect to the $\ell^2$-distance.

We now derive an upper bound on $B$ for our set $\mathcal{E}$. We note that for any $g \in \mathcal{P}_m$, we can apply Bessel's inequality to obtain
\[\left( \sum_i |\langle \phi_i, g \gamma_K(\lambda^{-1} g) \rangle |^2\right)^{\frac{1}{2}} \leq \left( \int_\mathbb{T} |g|^2 |\gamma_K(\lambda^{-1}g)|^2\right)^{\frac{1}{2}}.\]
Using the fact that $|\phi_i |\leq A$, we have that for all $x \in \mathbb{T}$, $|g(x)| \ll \sum_i |a_i| \leq \sqrt{m}$ by the Cauchy-Schwarz inequality.
Thus, since $||g||_{L^2} \leq 1$, we have $B \ll \gamma_K\left(\frac{\sqrt{m}}{\lambda}\right)$.

Next we address the quantity $N_2(\mathcal{E},t)$. For arbitrary $g, h \in \mathcal{P}_m$, we consider the quantity
\[\left( \sum_i \left( |\langle \phi_i, g \gamma_K(\lambda^{-1}g)\rangle| -|\langle \phi_i, h \gamma_K(\lambda^{-1} h)\rangle |\right)^2 \right)^{\frac{1}{2}} \leq \left( \sum_i |\langle \phi_i, g\gamma_K(\lambda^{-1} g) - h \gamma_K(\lambda^{-1} h)\rangle |^2\right)^{\frac{1}{2}},\]
using the fact that for any real numbers $a$ and $b$, $\left| |a| - |b|\right| \leq |a-b|$.
By Bessel's inequality, this quantity is $\leq || g \gamma_K(\lambda^{-1} g) - h \gamma_K(\lambda^{-1} h)||_{L^2}$.

Applying Lemma \ref{lem:gammabasic}, we see this is
\[ \ll \left|\left| |g-h| \cdot |\gamma_K(\lambda^{-1}g) + \gamma_K(\lambda^{-1} h)|\right|\right|_{L^2} \leq \lambda^{-(p-2)} \left|\left| |g-h| \cdot (|g|^{p-2} + |h|^{p-2})\right|\right|_{L^2}.\]
Noting that $|g|^{p-2} + |h|^{p-2} \leq 2 (|g|+|h|)^{p-2}$, this is $\ll \lambda^{-(p-2)} \left|\left| |g-h| \cdot (|g|+|h|)^{p-2}\right|\right|_{L^2}$. Now applying H\"{o}lder's inequality with conjugate exponents $\frac{1}{3 -p}$ and $\frac{1}{p-2}$, we see this quantity is
\[ \ll \lambda^{-(p-2)} \left( \int_\mathbb{T} |g-h|^{\frac{2}{3-p}}\right)^{\frac{3-p}{2}} \cdot \left( \int_\mathbb{T} (|g| + |h|)^2\right)^{\frac{p-2}{2}} \ll \lambda^{-(p-2)} \left|\left| g-h \right|\right|_{\frac{2}{3-p}}.\]

By a change of variable, we then have
\begin{equation}\label{Bstep3}
\int_0^\infty \left( \log N_2(\mathcal{E},t)\right)^{\frac{1}{2}} dt \ll \lambda^{-(p-2)} \int_0^\infty \left(\log N_{\frac{2}{3-p}}(\mathcal{P}_m,z)\right)^{\frac{1}{2}} dz.
\end{equation}
Here, $N_{\frac{2}{3-p}}(\mathcal{P}_m,z)$ denote the corresponding entropy numbers of $\mathcal{P}_m$ considered as a subset of the space $L^{\frac{2}{3-p}}(\mathbb{T},\mu)$.

Applying Lemma \ref{lem:entropy}, we obtain
\[
\int_0^\infty \left(\log N_{\frac{2}{3-p}}(\mathcal{P}_m, t)\right)^{\frac{1}{2}} dt \leq C_p \sqrt{m} \left(\log\left(\frac{N}{m} +1\right)\right)^{\frac{1}{2}} \left( \int_0^{\frac{1}{2}} \sqrt{\log\left(\frac{1}{t}\right)} + \int_{\frac{1}{2}}^\infty t^{-\frac{\nu_p}{2}} dt\right),
\]
where $C_p$ and $\nu_p$ denote values that depend only on $p$. We note that $\nu_p > 2$ for $2 < p <3$.
Thus, we deduce
\begin{equation}\label{Bstep4}
\int_0^\infty \left(\log N_{\frac{2}{3-p}}(\mathcal{P}_m, t)\right)^{\frac{1}{2}} dt \leq C_p \sqrt{m}\left(\log\left(\frac{N}{m}\right)+1\right)^{\frac{1}{2}}
\end{equation}
for some constant $C_p$ depending only on $p$ (we have abused notation a bit here, as this is not the same $C_p$ as in the previous statement).

Combining (\ref{Bstep3}) and (\ref{Bstep4}), we see that
\begin{equation}\label{Bstep5}
\int_0^\infty \left(\log N_2(\mathcal{E},t)\right)^{\frac{1}{2}} dt \leq C_p \cdot \lambda^{-(p-2)} \cdot \sqrt{m} \left(\log\left(\frac{N}{m}+1\right)\right)^{\frac{1}{2}}.
\end{equation}
We can now use this to obtain the following bound on the righthand side of (\ref{Bstep2}):
\[ \ll \left(\delta m + \frac{q_0}{\log\left(\frac{1}{\delta}\right)} \right)^{\frac{1}{2}} \gamma_K\left(\frac{\sqrt{m}}{\lambda}\right) + C_p \cdot \lambda^{-(p-2)}\cdot \sqrt{m} \left( \log\left(\frac{1}{\delta}\right)\right)^{-\frac{1}{2}} \left(\log\left( \frac{N}{m} +1\right)\right)^{\frac{1}{2}}.\]

Combining this with (\ref{Bstep1}), we conclude that
\[\left|\left| \sup_{\substack{|W| \leq m\\ |c_i| < 2/m}} \int_\mathbb{T} \Gamma_K\left(\frac{\sum_{i \in W} c_i \xi_i(\omega) \phi_i}{\lambda}\right) \right|\right|_{L^{q_0} (d\omega)}\]\[ \ll \lambda^{-2} \left(\delta + \frac{q_0}{m \log\left(\frac{1}{\delta}\right)}\right)^{\frac{1}{2}} \gamma_K\left(\frac{\sqrt{m}}{\lambda}\right) + C_p \cdot  \lambda^{-p} \left( \frac{\log\left(\frac{N}{m}+1\right)}{\log\left(\frac{1}{\delta}\right)}\right)^{\frac{1}{2}}.\]
This quantity will be $\leq 1$ for a choice of $\lambda$ that is
\[ \lambda \ll_p \delta^{\frac{1}{4}} K^{\frac{p-2}{2}} + \left(\frac{q_0}{\log\left(\frac{1}{\delta}\right)}\right)^{\frac{1}{4}} + \left( \frac{\log N}{\log\left(\frac{1}{\delta}\right)}\right)^{\frac{1}{2p}}.\]
\end{proof}

\subsection{Moving to General Coefficients}

\begin{lemma}\label{lem:Bgenco}
Let $\phi_1, \ldots, \phi_N$ be orthonormal functions on a probability space $\mathbb{T}, \mu$ satisfying $|\phi_i| \leq A$ everywhere on $\mathbb{T}$ for all $i$. Let $0 < \delta < 1$ and let $\{\xi_i\}_{i=1}^N$ be independent, $\{0,1\}$-valued random variables of mean $\delta$ on a probability space $\Omega$. Let $1 < K < \infty$, $2 < p <3$, $1 \leq M \leq N$ and $q_0 \geq 1$.
Then there exists a
\[\lambda \ll_p \delta^{\frac{1}{4}} K^{\frac{p-2}{2}} \sqrt{\log M} + \left(\frac{q_0+ \log \log M}{\log (\frac{1}{\delta})}\right)^{\frac{1}{4}}\sqrt{\log M} + \left(\frac{\log N}{\log ( \frac{1}{\delta})}\right)^{\frac{1}{2p}}\sqrt{\log M}\]
such that
\[ \left|\left| \sup_{\substack{||\{a_i\}||_{\ell^2} \leq 1 \\ |support(\{a_i\})| \leq M}} \int_\mathbb{T} \Gamma_K\left(\frac{\sum_{i} a_i \xi_i(\omega) \phi_i}{\lambda}\right) \right|\right|_{L^{q_0} (d\omega)} \leq 1.\]
\end{lemma}

\begin{proof}
For any fixed values $\{a_i\}$ with support size $\leq M$ such that $\sum_i a_i^2 \leq 1$, we divide them into $\log M$ levels corresponding to powers of 2. More precisely, for each $m$ that is a power of 2 that is $\geq 1$ and $<M$, we let $S_m$ denote the set of indices $i$ such that $1/m \leq |a_i|^2 < 2/m$. For $m = 2^{\lceil \log M \rceil}$, we define $S_m$ to be the set of all indices $i$ such that $|a_i|^2 < 2/m$. We then have that $[N]$ is the disjoint union of $S_m$ as $\log m$ ranges from 1 to $\lceil \log M \rceil$. Note that $|S_m| \leq m$ for each $m$. We also define $d_m = \sum_{i \in S_m} |a_i|^2$ for each $m$. We note that $\sum_m d_m \leq 1$.
For notational convenience, we also define the quantity $D := \sum_m \sqrt{d_m}$.

We then observe (for a fixed $\omega \in \Omega$)
\begin{eqnarray*}
\int_\mathbb{T} \Gamma_K\left(\frac{\sum_{i} a_i \xi_i(\omega) \phi_i}{\lambda}\right) & = & \int_\mathbb{T} \Gamma_K \left( \lambda^{-1} \sum_m \sum_{i \in S_m} a_i \xi_i(\omega) \phi_i \right) \\
&= &\int_\mathbb{T} \Gamma_K \left( \sum_m \frac{\sqrt{d_m}}{D} \left( \frac{D}{\lambda \sqrt{d_m}}\sum_{i \in S_m} a_i \xi_i(\omega) \phi_i \right) \right).
\end{eqnarray*}
Appealing to the convexity of $\Gamma_K$ and the linearity of the integral, we see this is
\[ \leq D^{-1} \sum_m \sqrt{d_m} \int_\mathbb{T} \Gamma_K \left( \frac{D}{\lambda \sqrt{d_m}} \sum_{i \in S_m} a_i \xi_i(\omega) \phi_i \right).\]

For each $m$, we define the random variable $V_{D,m}$ as:
\[V_{D,m}(\omega) = \sup_{\substack{|A|\leq m,\\ |c_i|^2 \leq 2/m}} \int_\mathbb{T} \Gamma_K \left( \frac{D}{\lambda} \sum_{i \in A} c_i \xi_i(\omega) \phi_i\right).\]
We now have
\[ \left|\left| \sup_{\substack{||\{a_i\}||_{\ell^2} \leq 1\\ |support(\{a_i\})| \leq M}} \int_\mathbb{T} \Gamma_K\left(\frac{\sum_{i} a_i \xi_i(\omega) \phi_i}{\lambda}\right) \right|\right|_{L^{q_0} (d\omega)}\]\[ \leq \left| \left| \sup_{||d_m||_{\ell^2 \leq 1}} D^{-1} \sum_{m} \sqrt{d_m} V_{D,m} \right|\right|_{L^{q_o}(d\omega)}.\]

We note that $D = \sum_m \sqrt{d_m} \ll \sqrt{\log M}$ by the Cauchy-Schwarz inequality (recall there are only $\lceil \log M \rceil$ values of $m$). Therefore, since $V_{D,m}$ is non-decreasing as a function of $D$, the quantity above is
\[ \ll \left|\left| \sup_{||d_m||_{\ell^2 \leq 1}} D^{-1} \left(\sum_m \sqrt{d_m}\right) \sup_m V_{D,m} \right|\right|_{L^{q_0}(d\omega)} \leq \left|\left| \sup_m V_{\sqrt{\log M},m} \right|\right|_{L^{q_o}(d\omega)}.\]
Applying Lemma \ref{lem:french}, we see that this is
\[ \ll \sup_{m} \left|\left| \sup_{\substack{|A|\leq m,\\ |c_i|^2 \leq 2/m}} \int_\mathbb{T} \Gamma_K \left( \frac{\sqrt{\log M}}{\lambda} \sum_{i \in A} c_i \xi_i(\omega) \phi_i\right) \right|\right|_{L^{q_o+\log \log M}(d\omega)} . \]

By setting
\[\lambda \ll_p \delta^{\frac{1}{4}} K^{\frac{p-2}{2}} \sqrt{\log M} + \left(\frac{q_0+ \log \log M}{\log (\frac{1}{\delta})}\right)^{\frac{1}{4}}\sqrt{\log M} + \left(\frac{\log N}{\log ( \frac{1}{\delta})}\right)^{\frac{1}{2p}}\sqrt{\log M}\]
and applying Lemma \ref{lem:Bbasic}, we obtain the result.
\end{proof}

\subsection{Obtaining a Good Partition}

\begin{lemma}\label{lem:Lpart} Let $\phi_1, \ldots, \phi_N$ be orthonormal functions on a probability space $\mathbb{T}, \mu$ satisfying $|\phi_i| \leq A$ everywhere on $\mathbb{T}$ for all $i$. Let $1 \leq M \leq L \leq N$ and $\gamma > 1$. It holds with probability at least $1 - c N^{-\gamma}$ (for some universal constant $c$) that
a subset $I$ of $N$ of size $L$ chosen uniformly randomly satisfies
\begin{equation}\label{intervalReq}
 \sup_{\substack{||\{a_i\}||_{\ell^2} \leq 1 \\ \text{support}(\{a_i\}) \leq M }} \int_\mathbb{T} \Gamma_K \left(\frac{\sum_{i \in I} a_i \phi_{i}}{\lambda}\right) \leq 2
\end{equation}
for a choice of $\lambda$ that is
\[ \lambda \ll_{p,\gamma} \left(\frac{L}{N}\right)^{\frac{1}{4}} K^{\frac{p-2}{2}} \sqrt{\log M} + \left(\frac{\log N}{\log (\frac{N}{L})}\right)^{\frac{1}{4}}\sqrt{\log M} + \left(\frac{\log N}{\log ( \frac{N}{L})}\right)^{\frac{1}{2p}}\sqrt{\log M}.\]
\end{lemma}

\begin{proof}We fix a value of $\lambda$ suitable to apply Lemma \ref{lem:Bgenco} (with $\delta= L/N$ and $q_0= 2 \gamma \log(N)$ fixed). For a real value $t > 1$, we let $E_t$ denote the event that a set $S$ of size of $L$ selected uniformly at random from $[N]$ satisfies
\[  \sup_{\substack{||\{a_i\}||_{\ell^2} \leq 1 \\ \text{support}(\{a_i\}) \leq M }} \int_\mathbb{T} \Gamma_K \left(\frac{\sum_{ i \in S} a_i \phi_i}{\lambda} \right)\leq t.\]

Similarly, for independent selectors $\{\xi_i\}$ with mean $\delta = L/N$, we let $E_t'$ denote the event that
\[  \sup_{\substack{||\{a_i\}||_{\ell^2} \leq 1 \\ \text{support}(\{a_i\}) \leq M }} \int_\mathbb{T} \Gamma_K \left( \frac{ \sum_i a_i \xi_i(\omega) \phi_i}{\lambda} \right) \leq t.\]

We begin by obtaining a lower bound on the probability of the event $E_t'$ using Lemma \ref{lem:Bgenco}. We have by Chebyshev's inequality:
\[ \mathbb{P} \left[ \sup_{||\{a_i\}||_{\ell^2} \leq 1} \int_\mathbb{T} \Gamma_K \left( \frac{ \sum_i a_i \xi_i(\omega) \phi_i}{\lambda} \right) > t \right] \leq t^{-q_0}.\]
Hence, $\mathbb{P}[E'_t] \geq 1 - t^{-q_0}$ .

Next we observe the relationship between $\mathbb{P}[E_t]$ and $\mathbb{P}[E'_t]$. By applying Lemma \ref{lem:SelToSet} to the complements of the events $E_t$ and $E'_t$, we have $1 -\mathbb{P}[E_t] \ll \sqrt{L}\cdot  (1-\mathbb{P}[E'_t])$. This implies that
\[ \mathbb{P}[E_t] \geq  1 - C \sqrt{L}\cdot t^{-q_0} \geq 1 - C \sqrt{N} \cdot t^{-2 \gamma \log(N)}\]
for some positive constant $C$.

Thus for a uniformly random subset of size $L$, the probability of the event $E_{2}$, that is that (\ref{intervalReq}) holds, is at least $1 - c N^{-\gamma}$.

\end{proof}

\begin{corollary}\label{cor:Lpart1} Let $\{\phi_i\}_{i=1}^N$ be a bounded orthonormal system. Let $1 \leq L \leq N$ and $\gamma > 1$. It holds with probability at least $1 - c N^{-\gamma}$ (for some universal constant $c$) that
a subset of $S \subseteq [N]$ of size $L$ chosen uniformly randomly satisfies:
\[ \sup_{\substack{||\{a_i\}||_{\ell^2} \leq 1 \\ \text{support}(\{a_i\}) \leq M }} \left|\left| \sum_{i \in S} a_i \phi_i  \right|\right|_{\Gamma_{(N/L)^{1/(2p-4)} }} \ll_{p,\gamma}  \left(\frac{\log N}{\log (\frac{N}{L})}\right)^{\frac{1}{4}}\sqrt{\log M}   \]
for all $M$ in the range $1\leq M \leq L$.

\end{corollary}
\begin{proof}
We apply the previous lemma with parameters $\gamma+1$ and $K:= \left(\frac{N}{L}\right)^{\frac{1}{2p-4}}$ and employ a union bound over the $\leq N$ values of $M$. We note that once we have a $\lambda$ such that
\[\int \Gamma_K\left(\frac{\sum_{i \in S} a_i \phi_i}{\lambda}\right) \leq 2,\]
we can multiply $\lambda$ by a constant to obtain an upper bound on the $\Gamma_K$-norm.
\end{proof}

\begin{corollary}\label{cor:Lpart2} Let $\{\phi_i\}_{i=1}^N$ be a bounded orthonormal system. For any $\gamma > 1$, it holds with probability at least $1- c N^{-\gamma}$ (for some universal $c$) that a uniformly randomly selected permutation $\sigma : [N] \rightarrow [N]$ satisfies the following.  If $I$ is an subinterval in $[N]$ then (for all $M \leq |I|$) we have
\[ \sup_{\substack{||\{a_i\}||_{\ell^2} \leq 1 \\ \text{support}(\{a_i\}) \leq M }} \left|\left| \sum_{i \in I} a_n \phi_{\sigma(n)}   \right|\right|_{\Gamma_{(N/|I|)^{1/(2p-4)} }} \ll_{p,\gamma}  \left(\frac{\log N}{\log (\frac{N}{|I|})}\right)^{\frac{1}{4}}\sqrt{\log M}   \]
(uniformly in the choice of interval $I$ and $M$).
\end{corollary}
\begin{proof}
We apply the previous corollary with the parameter $\gamma+2$ for each interval $I$ with $L:= |I|$. We then employ the union bound over the $\sim N^2$ intervals. Note that for each $I$, the image of $I$ under a randomly chosen permutation is equivalent to a randomly chosen subset of size $|I|$ inside $[N]$.

\end{proof}

\section{Bourgain's Theorem}

The estimates in the previous section are not strong enough to deduce Theorem \ref{thm:MainI}, however they do allow us to reprove Bourgain's theorem and show that it holds with large probability, a fact we will require later.

\begin{proposition}\label{prop:bourgain}Let $\{\phi_n\}_{n=1}^{N}$ denote a bounded orthonormal system. For any $\gamma > 1$ it holds with probability at least $1- c N^{-\gamma}$ (for some universal $c$) that a uniformly randomly selected permutation $\sigma : [N] \rightarrow [N]$ satisfies
\[||\mathcal{M}_{\sigma}f ||_{2} \ll_\gamma \log \log(N) ||f||_{2} \]
where $\mathcal{M}_{\sigma} f (x):= \max_{ \ell \leq N } \left|\sum_{ n=1}^{\ell} a_n \phi_{\sigma(n)} (x)\right|$ for all $f= \sum_{n\in [N]} a_n \phi_{\sigma(n)}$ in the span of the system.
\end{proposition}

\begin{proof}
We select a permutation $\sigma$ satisfying the conclusion of Corollary \ref{cor:Lpart2}.
We fix $f= \sum_{n\in [N]} a_n \phi_{\sigma(n)}$ and we consider intervals $I_{k,s}$ and points $i_{k,s}$ in a corresponding mass decomposition of $[N]$. For each $k$, we define $I_k^a$ to be the collection of intervals $I_{k,s}$ such that $|I_{k,s}|\leq 2^{-k/2}N$ and $I_k^b$ to be the collection of intervals $I_{k,s}$ such that $|I_{k,s}| > 2^{-k/2}N$. We observe that $|I^{b}_{k}| \leq 2^{k/2}$.

At each fixed point $x \in \mathbb{T}$, the value of $\mathcal{M}_{\sigma}f(x)$ is achieved on some subinterval of $[N]$ that can be decomposed into a union of dyadic intervals $I_{k,s}$ and points $i_{k,s}$ such that there is at most one interval $I_{k,s}$ for each value of $k$. We let $k^*:= \lceil 15 \log \log (N)\rceil$.
We then have the pointwise inequality
\begin{eqnarray*}\nonumber
  \mathcal{M}_{\sigma}f(x) & \ll & \sum_{k} \left( \sum_{\substack{s \\ I_{k,s} \in I_k^{b}   }} |f_{I_{k,s}}|^2\right)^{1/2} +  \sum_{k\leq k^*} \left( \sum_{\substack{s \\ I_{k,s} \in I_k^{a}   }} |f_{I_{k,s}}|^{2} \right)^{1/2}  \\ \nonumber
   & & +   \sum_{k} \max_{s} |f_{i_{k,s}}| +     \left(\sum_{\substack{1 \leq s \leq  2^{k^*} \\ I_{k^*, s} \in I^a_{k^*}}} |\tilde{f}_{I_{k^*,s}}|^p\right)^{1/p}.
\end{eqnarray*}
To see this, note that $\left(\sum_{I_{k,s} \in I_k^b} |f_{I_{k,s}}|^2\right)^{\frac{1}{2}}$ is an upper bound on the largest value of $|f_{I_{k,s}}|$ over all $s$ such that $I_{k,s} \in I_k^b$ for each $k$ (for example), and the final term above captures the contribution from dyadic subintervals for values of $k > k^*$.

We now consider the $L^2$ norm of each of these terms. Recall that $\int |f_{I_{k,s}}|^2 \leq 2^{-k}$ and $|I^{b}_{k}| \leq 2^{k/2}$, thus we have
\[ \left|\left|\sum_{k} \left( \sum_{\substack{s \\ I_{k,s} \in I^{b}_{k}   }} |f_{I_{k,s}}|^2\right)^{1/2} \right|\right|_{2} \leq \sum_{k} \left(\sum_{\substack{s \\ I_{k,s} \in I^{b}_k}} ||f_{I_{k,s}}||_{2}^2  \right)^{1/2} \leq \sum_{k} 2^{-k/4} \ll 1.\]
Next,
\[\left|\left| \sum_{k\leq k^*} \left( \sum_{\substack{s \\ I_{k,s} \in I^{a}_k   }} |f_{I_{k,s}}|^{2} \right)^{1/2} \right|\right|_{2} \leq \sum_{ k \leq k^*}\left|\left| \left( \sum_{1 \leq s \leq 2^{k}  } |f_{I_{k,s}}|^{2} \right)^{1/2} \right|\right|_{2} \ll \log \log(N).  \]
We also have that, using that $|\phi_n(x)|\leq A$ and $||f_{i_{k,s}}||_{2}^2 \leq 2^{-k+1}$, we have the pointwise bound
\[  \sum_{k} \max_{s} |f_{i_{k,s}}|  \ll \sum_{k} 2^{-k/2} \ll  1.\]

Finally, we must consider the quantity
\[ \left|\left| \left(\sum_{\substack{1 \leq s \leq 2^{k^*}\\ I_{k^*,s} \in I^a_{k^*}}} |\tilde{f}_{I_{k^*,s}}|^p\right)^{1/p} \right|\right|_{2}.\]
Using Corollary \ref{cor:Lpart2} and Lemma \ref{lem:maxde}, we may decompose $\tilde{f}_{I_{k^*,s}} = G_{I_{k^*,s}} + E_{I_{k^*,s}}$ where $|| G_{I_{k^*,s}} ||_{p} \ll_{p,\gamma}  \log^{3/4}(N) || f_{I_{k^*,s}} ||_{2} $ and $|| E_{I_{k^*,s}}||_{2}  \ll_{p,\gamma}  K^{(2-p)/2 }  \log^{7/4}(N) || f_{I_{k^*,s}} ||_{2}    $
where $K = \left( \frac{N}{|I_{k^*,s}|}\right)^{\frac{1}{2p-4}}$. For $I_{k^*,s} \in I_{k^*}^a$, we have $|I_{k^*,s}| \leq 2^{-k^*/2} N$, so $K \geq 2^{\frac{k^*}{4(p-2)}}$.


Setting $p = 5/2$, we now have
\[ \left|\left| \left(\sum_{1 \leq s \leq  2^{k^*}} |G_{I_{k^*,s}}|^p\right)^{1/p} \right|\right|_{2} \leq
\left|\left| \left(\sum_{1 \leq s \leq  2^{k^*}} |G_{I_{k^*,s}}|^p\right)^{1/p} \right|\right|_{p} = \left(\sum_{1 \leq s \leq  2^{k^*}}  || G_{I_{k^*,s}} ||_{5/2}^{5/2} \right)^{2/5} \]
\[  \ll  \log^{3/4}(N) \left(\sum_{1 \leq s \leq  2^{k^*}}  || f_{I_{k^*,s}} ||_{2}^{5/2} \right)^{2/5}. \]
Now $\sum_{1 \leq s \leq  2^{k^*}} || f_{I_{k^*,s}} ||_{2}^{5/2} \ll 2^{k^*} \left( 2^{-k^*/2}\right)^{5/2} \ll \log^{-15/4}(N) $, which implies the quantity  is
\[ \ll \log^{3/4}(N) \log^{-3/2}(N) \ll 1.\]

We last consider
\[\left|\left| \left(\sum_{\substack{1 \leq s \leq  2^{k^*}\\ I_{k^*,s} \in I^a_{k^*}}} |E_{I_{k^*,s}}|^p\right)^{1/p} \right|\right|_{2} \leq \left( \int \sum_{\substack{1 \leq s \leq  2^{k^*}\\ I_{k^*,s} \in I^a_{k^*}}} |E_{I_{k^*,s}}|^2\right)^{\frac{1}{2}}\]
\[\ll \left( \sum_{\substack{1 \leq s \leq 2^{k^*}\\ I^{k^*,s} \in I^a_{k^*}}} \log^{7/2}(N) K^{2-p} ||f_{I_{k^*,s}}||_2^2\right)^{\frac{1}{2}}.\]
We recall that $K = \left(\frac{N}{|I_{k^*,s}|}\right)^{\frac{1}{2(p-2)}}$, so for $I_{k^*,s} \in I^a_{k^*}$, we have $K \geq 2^{\frac{k^*}{4(p-2)}}$. Thus, $K^{\frac{2-p}{2}} \leq 2^{-\frac{k^*}{8}}$, so the quantity above is $\ll \log^{7/4}(N) \log^{-15/8}(N) \ll 1$.
This completes the proof.
\end{proof}

\section{Improving the Bound by Passing to Random Subsets}

\subsection{Bounding the Expectation}
For a fixed interval $I \subseteq [N]$ and a bounded ONS $\{\phi_i\}_{i=1}^N$, we define the operator $U_I$ from $\ell^2_N$ to the Banach space of real valued functions on $\mathbb{T}$ with norm $||\cdot ||_{\Gamma_K}$ by mapping a sequence $\{a_i\}_{i=1}^N$ to the function $\sum_{i \in I} a_i \phi_i$.

\begin{lemma}\label{lem:TalApp} Let $\{\phi_i\}_{i=1}^N$ be a bounded ONS. Let $I$ be a fixed interval satisfying
\begin{equation}\label{Tal1}
\sup_{||\{a_i\}||_{\ell^2} \leq 1} \left|\left| \sum_{i \in I} a_i \phi_i\right|\right|_{\Gamma_{K}} \leq \lambda
\end{equation}
for a fixed $\lambda$ and $K$. Let $\{\xi_i\}_{i\in I}$ be independent selectors with mean $\nu$. We let $S \subseteq I$ denote the set of indices $i \in I$ such that $\xi_i = 1$. Then
\[ \mathbb{E}\left[ \sup_{||\{a_i\}||_{\ell^2} \leq 1} \left|\left| \sum_{i \in S} a_i \phi_i \right|\right|_{\Gamma_K}^2 \right] \ll \frac{\lambda^2 + C_p \log N}{ \log \left(\frac{1}{\nu}\right)}.\]
\end{lemma}

\begin{proof} We apply Lemma \ref{lem:Talagrand} to the operator $U_I$. This yields
\[
\mathbb{E}\left[ \sup_{||\{a_i\}||_{\ell^2} \leq 1} \left|\left| \sum_{i \in S} a_i \phi_i \right|\right|_{\Gamma_K}^2 \right] \ll \left(\log \left(\frac{1}{\nu}\right)\right)^{-1} \left(C_p \log N + \sup_{||\{a_i\}||_{\ell^2} \leq 1} \left|\left| \sum_{i \in I} a_i \phi_i \right|\right|_{\Gamma_K}^2\right).
\]
Using the interval $I$ satisfies (\ref{Tal1}), we see this quantity is $\ll \frac{\lambda^2 + C_p \log N }{ \log \left(\frac{1}{\nu}\right)}.$

\end{proof}

\subsection{Bounding the Concentration}
We now consider the concentration of the random variable $\sup_{||\{a_i\}||_{\ell^2} \leq 1} \left|\left| \sum_{i \in S} a_i \phi_i \right|\right|_{\Gamma_K}$ around its expectation.

\begin{lemma}\label{lem:restest}  Let $\{\phi_i\}_{i=1}^N$ be a bounded ONS. Fix $0 < \beta < 1$ and let $L \leq N$, $\delta := L/N$. Then for any $\gamma >1 $ it holds with probability at least $1- c N^{-\gamma}$ (for some universal $c$) that a uniformly randomly selected subset $S$ of $[N]$ of size $L$ satisfies
\begin{equation}
\sup_{||\{a_i\}||_{\ell^2} \leq 1} \left|\left| \sum_{i \in S} a_i \phi_i\right|\right|_{\Gamma_{K}} \ll_{p,\gamma} \frac{\log^{3/4}(N)}{\log^{3/4}(\frac{N}{L})}
\end{equation}
for $K:= \min \left\{ \left(\frac{1}{\delta}\right)^{1/(4p-8)}, \log^{\frac{7}{2p-4}}(N)\right\} $ and $2 < p <3$.
\end{lemma}

\begin{proof}
We will pass to a set of size $L$ in two stages. We first define an intermediate value $L_1\geq L$ such that $\frac{N}{L_1} = \frac{L_1}{L}$. We will set $K:= \min \left\{ \left(\frac{N}{L_1}\right)^{1/(2p-2)}, \log^{\frac{7}{2p-4}}(N)\right\}$. We then apply Corollary \ref{cor:Lpart1} to conclude that with sufficiently high probability, say at least $1-cN^{-(\gamma+2)}$, a random subset $W$ of size $L_1$ inside $[N]$ will satisfy
\begin{equation}\label{eq:lambdabound}
\sup_{||\{a_i\}||_{\ell^2} \leq 1} \left|\left| \sum_{i \in W} a_i \phi_i\right|\right|_{\Gamma_{K}} \ll_{p,\gamma} \left(\frac{ \log N}{\log \left(\frac{N}{L_1}\right)} \right)^{\frac{1}{4}} \cdot \sqrt{\log L_1}.
\end{equation}

We now condition on the above result, and consider choosing a random subset $S \subseteq W$ of size $L$.
We first note that for any such set $S$,
\[ \left|\left| \sum_{i \in S} a_i \phi_i \right|\right|_{\Gamma_K} = \sup_{ ||g||_{\Gamma^*} \leq 1} \langle \sum_{i \in S} a_i \phi_i, g\rangle,\]
where $||\cdot ||_{\Gamma^*}$ denotes the dual norm to $||\cdot ||_{\Gamma_K}$.
Using the Cauchy-Schwarz inequality, we then see that
\begin{equation}\label{eq:split1}
\sup_{||\{a_i\}||_{\ell^2} \leq 1} \left|\left| \sum_{i \in S} a_i \phi_i \right|\right|_{\Gamma_K} = \sup_{||\{a_i\}||_{\ell^2} \leq 1} \sup_{ ||g||_{\Gamma^*} \leq 1} \sum_{i \in S} a_i \langle \phi_i, g \rangle \leq \sup_{ ||g||_{\Gamma^*} \leq 1} \left( \sum_{i \in S} \left| \langle \phi_i, g\rangle \right|^2 \right)^{\frac{1}{2}}.
\end{equation}
In fact, noting that one can set $a_i = \frac{\langle \phi_i g \rangle}{\sqrt{\sum_{i \in S} \langle \phi_i, g \rangle^2}}$, we see that the inequality in (\ref{eq:split1}) is an equality.

We let $C$ be a threshold parameter that we will specify later. We define the functions $\chi_1, \chi_2 : \mathbb{R} \rightarrow \mathbb{R}$ as follows:
\[\chi_1(x) = \left\{
    \begin{array}{ll}
      x, & \hbox{if $|x| \leq C$;} \\
      0, & \hbox{otherwise}
    \end{array}
  \right. , \; \;
\chi_2(x) = \left\{
    \begin{array}{ll}
      0, & \hbox{if $|x| \leq C$;} \\
      x, & \hbox{otherwise}
    \end{array}
  \right.\]
We then have
\begin{equation}\label{eq:split2}
\sup_{ ||g||_{\Gamma^*} \leq 1}  \sum_{i \in S} \left| \langle \phi_i, g\rangle \right|^2 \leq \sup_{ ||g||_{\Gamma^*} \leq 1} \sum_{i \in S} \chi_1 \left( \left| \langle \phi_i, g\rangle \right|^2 \right) + \sup_{||h||_{\Gamma^*} \leq 1} \sum_{i \in S} \chi_2 \left( \left| \langle \phi_i, h\rangle \right|^2 \right).
\end{equation}

We will deal separately with the two quantities in (\ref{eq:split2}). To address the first quantity, we will employ Lemma \ref{lem:Ledoux}. In this case, our random variables $Y_1, \ldots, Y_N$ are defined as follows. We let $\Omega$ denote the probability space of the independent selectors $\{\xi_i\}_{i \in I}$, each with mean $\nu:= \frac{L}{L_1}$. Then each $\omega \in \Omega$ is associated to a subset $S \subseteq I$. (This distribution of $S$ differs of course from selecting a set of size exactly $L$, but we will analyze the relevant quantity in this case first and then derive a bound for the case of fixed size.) We define $Y_i(\omega)$ to be equal to $\phi_i$ when $i \in S$, and equal to 0 otherwise (more formally, the constant zero function from $\mathbb{T}$ to $\mathbb{C}$).

For notational convenience, we defined the random variable $Z:\Omega \rightarrow \mathbb{R}$ by
\[ Z(\omega) = \sup_{ ||g||_{\Gamma^*} \leq 1} \sum_{i \in S} \chi_1 \left( \left| \langle \phi_i, g\rangle \right|^2 \right),\]
where $S$ is determined from $\omega$ as described above.
Applying Lemma \ref{lem:Ledoux}, it follows that (for every positive real number $\tau$)
\begin{equation}\label{eq:Ledoux}
\mathbb{P}[ |Z- \mathbb{E}[Z]|\geq \tau] \leq 3 \exp \left( \frac{-\tau}{\kappa C} \log \left( 1 + \frac{C\tau}{\mathbb{E}[\sigma^2]}\right)\right),
\end{equation}
where $\kappa$ is a positive constant, and $\sigma^2 = \sup_{||g||_{\Gamma^*}\leq 1} \sum_{i \in S} \left| \chi_1 \left( \left| \langle \phi_i, g\rangle \right|^2 \right)\right|^2$.

Recalling our definition of $Z$, we observe
\[ Z(\omega) \leq \sup_{ ||g||_{\Gamma^*} \leq 1}  \sum_{i \in S} \left| \langle \phi_i, g\rangle \right|^2 =  \sup_{||\{a_i\}||_{\ell^2} \leq 1} \left|\left| \sum_{i \in S} a_i \phi_i \right|\right|_{\Gamma_K}^2.\]
(We have used here that (\ref{eq:split1}) is actually an equality.)
Consequently,
\[\mathbb{E} [Z] \leq \mathbb{E}\left[ \sup_{||\{a_i\}||_{\ell^2} \leq 1} \left|\left| \sum_{i \in S} a_i \phi_i \right|\right|_{\Gamma_K}^2 \right].\]
Applying Lemma \ref{lem:TalApp}, we have
\[ \mathbb{E}[Z] \ll \frac{\lambda^2 + C_p \log N }{ \log \left(\frac{1}{\nu}\right)},\]
where with sufficient probability we have $\lambda \ll_{p, \gamma} \log^{\frac{1}{4}}(N) \log^{-\frac{1}{4}}\left( \frac{N}{L_1}\right) \sqrt{\log L_1}$ from (\ref{eq:lambdabound}).
Thus, there exist a constant $A_1$ such that
\begin{equation}\label{eq:split3}
\mathbb{P}\left[ Z\geq \tau+ A_1\cdot \frac{\lambda^2 + C_p \log N}{\log \left(\frac{1}{\nu}\right)}\right] \leq 3 \exp \left( \frac{-\tau}{\kappa C} \log \left( 1 + \frac{C\tau}{\mathbb{E}[\sigma^2]}\right)\right).
\end{equation}

It remains to bound the quantity
\[\mathbb{E}[\sigma^2] = \mathbb{E}\left[ \sup_{||g||_{\Gamma^*}\leq 1} \sum_{i \in S} \left| \chi_1 \left( \left| \langle \phi_i, g\rangle \right|^2 \right)\right|^2\right].\]
Employing line (7.10) on p. 141 of \cite{Ledoux}, we see this is
\begin{equation}\label{eq:sigma1}
 \ll \sup_{||g||_{\Gamma^*}\leq 1} \sum_{i \in I} \mathbb{E} \left[ \left(\chi_1\left(|\langle \xi_i \phi_i, g\rangle|^2\right)\right)^2\right] + C \mathbb{E}\left[ \sup_{||g||_{\Gamma^*}\leq 1} \sum_{i \in S} \chi_1\left( |\langle \phi_i, g\rangle|^2\right) \right].
\end{equation}
By removing the cutoff at $C$, the latter quantity in (\ref{eq:sigma1}) is
\[ \leq C \mathbb{E}\left[ \sup_{||g||_{\Gamma^*}\leq 1} \sum_{i \in S} |\langle \phi_i, g\rangle|^2\right] = C \mathbb{E}\left[ \sup_{||\{a_i\}||_{\ell^2}\leq 1} \left|\left| \sum_{i \in S} a_i \phi_i \right|\right|_{\Gamma_K}^2 \right]\]
\[\ll C \left(\frac{\lambda^2 + C_p \log N }{ \log \left(\frac{1}{\nu}\right)}\right)\]
where we have recalled that (\ref{eq:split1}) is an equality and have applied Lemma \ref{lem:TalApp}.

We now consider the first quantity in (\ref{eq:sigma1}). Using the definition of $\chi_1$ and Lemma \ref{lem:supportBound}, we see it is
\[\ll C  \sup_{||g||_{\Gamma^*}\leq 1} \sum_{i \in I} \mathbb{E} \left[  |\langle \xi_i \phi_i, g\rangle|^2 \right]
\ll C \nu \sup_{||g||_{\Gamma^*}\leq 1} ||g||_{2}^2 \ll C \nu p^2 K^{2p-2}.\]

Putting this together we have, for universal constants $A_1$ and $A_2$, that
$$\mathbb{P}\left[ Z\geq \tau+ A_1\cdot \frac{\lambda^2 + C_p \log N}{\log \left(\frac{1}{\nu}\right)}\right]$$
\begin{equation}\label{eq:concentrationwp}
\leq 3 \exp \left( \frac{-\tau}{\kappa C} \log \left( 1 + \frac{C\tau}{A_2 C \nu p^2 K^{2p-2} + A_2 C \left(\frac{\lambda^2 + C_p \log N }{ \log \left(\frac{1}{\nu}\right)}\right)}\right)\right).
\end{equation}
We now set $\tau = A_1\cdot \frac{\lambda^2 + C_p \log N}{\log \left(\frac{1}{\nu}\right)}$. Since $K^{2p-2} \leq \frac{1}{\nu}$, we see that the quantity inside the logarithm in (\ref{eq:concentrationwp}) above is at least $1 + \alpha \tau$, where $\alpha$ is a positive constant depending on $p, \gamma$.
Since $\nu \geq \frac{1}{N}$, it is clear that $\tau$ is bounded away from 0. Thus, it suffices to take $C = A_{\gamma, p} \log^{-1}(N)$ (for some constant $A_{\gamma,p}$ dependent only on $p$ and $\gamma$).

Using Lemma \ref{lem:SelToSet}, we can derive an analogous upper bound on the probability of the event that
\[\sup_{ ||g||_{\Gamma^*} \leq 1} \sum_{i \in S} \chi_1 \left( \left| \langle \phi_i, g\rangle \right|^2 \right) \geq 2\tau\]
for a randomly chosen set $S$ of size $L$ (inside $W$). The extra factor of $\leq \sqrt{N}$ can be easily accommodated by choosing $C$ to be a sufficiently high power of $\log N$.

We note that
\[ \tau = A_1\cdot \frac{\lambda^2 + C_p \log N}{\log \left(\frac{1}{\nu}\right)} \ll_{p,\gamma} \frac{\log^{3/2}(N)}{\log^{3/2}\left(\frac{1}{\nu}\right)} + \frac{\log N}{\log \left(\frac{1}{\nu}\right)} \]
whenever (\ref{eq:lambdabound}) holds. Thus, the square root of this quantity is $\ll_{p,\gamma} \frac{\log^{3/4}(N)}{\log^{3/4}\left(\frac{N}{L}\right)}$ as required.

We return to consider the second quantity in (\ref{eq:split2}). By Bessel's inequality, the number of nonzero terms appearing in the sum is at most $C^{-1} \sup_{||h||_{\Gamma^*} \leq 1} || h||_{L^2}^2$. From Lemma \ref{lem:supportBound} we see that this is $ C^{-1} p^2 K^{2p-2}$.

We then have
\begin{equation}\label{eq:split4}
\sup_{||h||_{\Gamma^*} \leq 1} \sum_{i \in S} \chi_2 \left( \left| \langle \phi_i, h\rangle \right|^2 \right)\leq \sup_{ ||h||_{\Gamma^*}\leq 1} \; \sup_{\substack{I \subseteq W\\ |I|\leq C^{-1} p^2 K^{2p-2}}} \; \sum_{i \in I} |\langle \phi_i, h\rangle |^2.
\end{equation}
We next observe that
\begin{equation}\label{eq:split5}
\sup_{ ||h||_{\Gamma^*}\leq 1} \; \sup_{\substack{I \subseteq W\\ |I|\leq C^{-1} p^2 K^{2p-2}}}  \sum_{i \in I} |\langle \phi_i,h\rangle |^2 \; \; \leq \sup_{\substack{||\{a_i\}||_{\ell^2}\leq 1 \\ |\text{Support}(\{a_i\})| \leq C^{-1} p^2 K^{2p-2}}} \sup_{||h||_{\Gamma^*}\leq 1} \left( \sum_{i \in W} a_i \langle \phi_i, h \rangle \right)^2.
\end{equation}
To see this, note that on can set $a_i := \frac{\langle \phi_i, h\rangle}{\left(\sum_{j \in I} |\langle \phi_j, h\rangle|^2\right)^{\frac{1}{2}}}$ for all $i$ in a set $I$ of size at most $C^{-1} p^2 K^{2p-2}$.

Combining (\ref{eq:split4}) and (\ref{eq:split5}), we see that
\[ \sup_{||h||_{\Gamma^*} \leq 1} \sum_{i \in S} \chi_2 \left( \left| \langle \phi_i, h\rangle \right|^2 \right) \leq \sup_{\substack{||\{a_i\}||_{\ell^2}\leq 1 \\ |\text{Support}(\{a_i\})| \leq C^{-1} p^2 K^{2p-2}}} \sup_{||h||_{\Gamma^*}\leq 1} \left(\sum_{i \in W} a_i \langle \phi_i, h \rangle \right)^2.\]
Employing Lemma \ref{lem:dual}, this is
\[ \ll \sup_{\substack{||\{a_i\}||_{\ell^2}\leq 1 \\ |\text{Support}(\{a_i\})| \leq C^{-1} p^2 K^{2p-2}}} \left|\left| \sum_{i \in W} a_i \phi_i \right|\right|_{\Gamma_K}^2.\]

By (our earlier application of) Corollary \ref{cor:Lpart1}, we see this is
$$\ll_{p, \gamma} \left(\frac{\log N}{\log \left(\frac{N}{L_1}\right)}\right)^{\frac{1}{2}} \log ( C^{-1} p^2 K^{2p-2}) = \left(\frac{\log N}{\log \left(\frac{N}{L_1}\right)}\right)^{\frac{1}{2}} \log ( A_{\gamma,p}^{-1} \log(N) p^2 K^{2p-2})  $$
\begin{equation}\label{eq:bourconc}
\ll_{p,\gamma} \left(\frac{\log N}{\log \left(\frac{N}{L_1}\right)}\right)^{\frac{1}{2}} \log\log(N).
\end{equation}
Clearly this term is acceptable and completes the proof.
\end{proof}

\subsection{Deriving the Main Theorem}

Using Lemma \ref{lem:restest} with Corollary \ref{cor:v2subinterval} allows us to take $r=3/4$ and $p=3-\epsilon'$ in Lemma \ref{lem:intlenapprox} for any $0 < \epsilon' <1$. More precisely, we can set $\theta := \frac{3/4}{1/2 + 3/4 - 1/(3-\epsilon')} = \frac{3}{5 - 4/(3-\epsilon')}$. To obtain the hypotheses needed to apply Lemma \ref{lem:intlenapprox}, we employ Corollary \ref{cor:v2subinterval}. To obtain the hypotheses needed to apply this corollary, we employ Lemma \ref{lem:restest} with $\delta := 2^{-2\log^{\theta}(N)}$. Note here that $K$ will be set so that $\log^{r+1}(N) K^{2-p/2} \ll 1$ as required. We also employ Proposition \ref{prop:bourgain}. Then, with $\epsilon$ defined by $\frac{9}{22} + \epsilon = \frac{3}{10- 8/(3-\epsilon')}$, we obtain:

\begin{thm1}Let $\Phi := \{ \phi_n \}_{n=1}^{N}$ denote a bounded ONS. Let $\epsilon >0 , \gamma >1$. Then, with probability at least $1- c N^{-\gamma}$ (for some universal $c$), for a uniformly random permutation $\sigma: [N] \rightarrow [N]$, the system $\{ \phi_\sigma(n) \}_{n=1}^{N}$ will satisfy
$$|| \mathcal{V}^2 f ||_{2} \ll_{\epsilon, \gamma} \log^{\frac{9}{22} + \epsilon}(N) ||f||_{2}.$$
\end{thm1}

\section{Concluding Remarks}\label{sec:remarks}
We remarked earlier that when proving the main probabilistic estimate, Lemma \ref{lem:restest}, we were unable to work exclusively in either Bourgain or Talagrand's frameworks, and had to use a hybrid of the two. We briefly expound on these issues and how our approach addresses them.

Bourgain's approach to the $\Lambda(p)$-problem makes very careful use of the structure of the $L^p$ norms, in particular relying on delicate pointwise inequalities involving the function $x \rightarrow |x|^p$. We have been unable to find appropriate analogs of these arguments in the case of the more awkward $\Gamma_{K}$ norms of relevance here. Many of these issues can be avoided by first pigeonholing the coefficients to a single level set (indeed even Bourgain's paper \cite{BourLp} can be substantially simplified in this case, see for instance the argument sketched in \cite{BourSur}) however this introduces a logarithmic factor that is fatal for the current application (although, this is essentially the idea behind the estimates in Section 4.1, and the prior work in the maximal and $\mathcal{V}^r$ cases).

In contrast, Talagrand's methods can be applied to a very general class of norms, including those considered here (as we have partly done via Lemma \ref{lem:Talagrand}). However, when used in the framework of \cite{TalagrandSmooth} (such as Theorem 1.2 there), one `loses' a factor of $N^{-\epsilon}$ in the density of the resulting subsets. In the case of $L^p$ norms and polynomially sparse sets (of relevance in the $\Lambda(p)$-problem), Talagrand is able to decompose the norm into two parts. On one part, better estimates are true for denser sets (and thus the loss of a $N^{-\epsilon}$ factor is acceptable) and his very general theorem can be applied. The other part can be handled with alternate (and more elementary) methods. Again, we have been unable to find an analog of these arguments in the setting of $\Gamma_K$ norms and on the denser sets of relevance here.

Our approach has been to use Bourgain's arguments to prove that the $\Gamma_K$ norm on a random subset of the relevant density is within a logarithmic factor of what is needed (we note that the arguments from section 3 of \cite{TalagrandSmooth} are too crude for the purposes here).  We then use the orthonormal system associated to this set as the starting point for the application of Talagrand's methods, which are invoked by passing to a slightly sparser random subset.

This approach gives satisfactory control over the expectation of the $\Gamma_K$ norm on the resulting subset, however our application requires more information about the concentration around the expectation.  Lemma \ref{lem:Ledoux} provides a useful estimate in this respect, however this alone does not seem sufficient for our application. Fortunately, this estimate is insufficient only when some of the coefficients are large and we are able to adequately control the contribution from terms with large coefficients by returning to Bourgain's argument with pigeonholing, since the number of relevant level sets is then reduced.
Likely, a better understanding/treatment of these issues will lead to an improved estimate.

\vspace{1cm}
Some final remarks:
\begin{remark} As with many orthonormal system results (for instance those of \cite{BourLp} and \cite{Bour}), our arguments can be easily modified to treat the more general case of Hilbertian systems.
\end{remark}

\begin{remark}
A quantitative form of Kolmogorov's rearrangement theorem due to Nakata \cite{Nakata} states that the first $N$ exponentials can be reordered such that there exists an $f$ in their span such that $|| \mathcal{M} f ||_{L^2}  \gg \log^{1/4}(N) ||f||_{L^2}$. From modulation and dilation symmetries, this holds for any set of $N$ exponentials associated to an arithmetic progression.  Notice that a random reordering of $[N]$ will contain an increasing arithmetic progression of length at least $ \gg \log^c(N)$  (for some absolute constant $c>0$) in this `bad' ordering. Thus while a random reordering decreases the norm of the $\mathcal{V}^2$ operator, it increases the norm of the $\mathcal{M}$ operator to (at least) $ (\log \log (N))^{1/4}$, with large probability.
\end{remark}

\begin{remark}
In a related paper \cite{LewkoV2span}, we have shown that given an arbitrary ONS of length $N$, one may find an alternate ONS that spans the same space such that $||\mathcal{V}^2f||_{L^2} \ll \sqrt{\log\log(N)} ||f||_{L^2}$, which is sharp.  Where the results in our current work rely on estimates for what may be called selector processes, there the problem can be reduced to estimates for Gaussian processes. Generally, estimates in that setting are stronger and better understood.
\end{remark}

\texttt{A. Lewko, Microsoft Research}

\textit{allew@microsoft.com}
\vspace*{0.5cm}

\texttt{M. Lewko, Department of Mathematics, University of California, Los Angeles}

\textit{mlewko@math.ucla.edu}


\begin{thebibliography}{10}

\bibitem{Bo}
B. Bollobas, Random Graphs, 2nd. Ed., Cambridge University Press (2001).

\bibitem{BourLp}
J. Bourgain, Bounded orthogonal systems and the $\Lambda(p)$--set problem. Acta Math. 162 (1989), no. 3--4, 227-–245.

\bibitem{BourSur}
J. Bourgain, $\Lambda(p)$-sets in analysis: results, problems and related aspects. Handbook of the geometry of Banach spaces, Vol. I, 195–232, North-Holland, Amsterdam, 2001.

\bibitem{Bour}
J. Bourgain, On Kolmogorov's rearrangement problem for orthogonal systems and Garsia's conjecture. Geometric aspects of functional analysis (1987--88),
Lecture Notes in Math., 1376, Springer, Berlin, (1989) 209-250.


\bibitem{GarsiaRe}
A. Garsia, Existence of almost everywhere convergent rearrangements for Fourier series of $L^2$ functions. Ann. of Math. (2) 79 (1964) 623–-629

\bibitem{GarsiaBook}
A. Garsia, Topics in almost everywhere convergence. Lectures in Advanced Mathematics, 4 Markham Publishing Co., Chicago, Ill. (1970)

\bibitem{KR}
M. Krasnoselskii, J. Rutickii, Convex functions and Orlicz spaces. Translated from the first Russian edition by Leo F. Boron. P. Noordhoff Ltd., Groningen (1961).

\bibitem{Ledoux}
M. Ledoux, The concentration of measure phenomenon. Mathematical Surveys and Monographs, 89. American Mathematical Society, Providence, RI, (2001)

\bibitem{LewkoProb}
A. Lewko and M. Lewko, An exact asymptotic for the square variation of partial sum processes, arXiv:1106.0783

\bibitem{LewkoO}
A. Lewko and M. Lewko, Estimates for the square variation of partial sums of Fourier series and their rearrangements. J. Funct. Anal. 262 (2012), no. 6, 2561–-2607


\bibitem{LewkoV2span}
A. Lewko and M. Lewko, Orthonormal systems in linear spans, arXiv:1205.2420

\bibitem{Nakata}
S. Nakata, On the divergence of rearranged Fourier series of square integrable functions, Acta Sci. Math. 32 (1971), 59--70.

\bibitem{varCarleson}
R. Oberlin, A. Seeger, T. Tao, C. Thiele, J. Wright, A variation norm Carleson theorem. J. Eur. Math. Soc. 14 (2012), no. 2, 421–-464

\bibitem{Olevskii}
A. Olevskii, Divergent series for complete systems in $L^2$. Dokl. Akad. Nauk SSSR 138 (1961) 545–-548

\bibitem{TalagrandProd}
M. Talagrand, New concentration inequalities in product spaces. Invent. Math. 126 (1996), no. 3, 505–-563

\bibitem{TalagrandSmooth}
M. Talagrand, Sections of smooth convex bodies via majorizing measures, Acta Math. 175 (1995), no. 2, 273-–300

\bibitem{TalagrandSelect}
M. Talagrand, Selecting a proportion of characters, Israel J. Math. 108 (1998), 173-–191

\bibitem{TalagrandGeneric}
M. Talagrand, The generic chaining. Upper and lower bounds of stochastic processes. Springer Monographs in Mathematics. Springer-Verlag, Berlin, (2005)

\bibitem{Taylor}
S. Taylor, Exact asymptotic estimates of Brownian path variation. Duke Math. J. 39 (1972), 219–-241

\end{thebibliography}
\end{document}